\documentclass[a4paper, 11pt]{amsart}

\usepackage{amsmath,amssymb,amsthm,amscd,enumerate,setspace}
\usepackage[all]{xy}
\usepackage[dvips]{graphicx}

\usepackage{hyperref}

\input xypic
\xyoption{all}

\newcommand{\rar}{\rightarrow}
\newcommand{\lar}{\longrightarrow}

\newcommand{\llar}{-\kern-5pt-\kern-5pt\longrightarrow}
\newcommand{\lllar}{-\kern-5pt-\kern-5pt\llar}

\newtheorem{Theorem}{Theorem}[section]
\newtheorem{Lemma}[Theorem]{Lemma}
\newtheorem{Corollary}[Theorem]{Corollary}
\newtheorem{Proposition}[Theorem]{Proposition}

\newtheorem{Conjecture}[Theorem]{Conjecture}

\newtheorem{Remark}[Theorem]{Remark}
\newtheorem{Example}[Theorem]{Example}
\newtheorem{Definition}[Theorem]{Definition}
\newtheorem{Question}[Theorem]{Question}

\def\Ass{\mbox{\rm Ass}}

\def\ds{\displaystyle}

\def\gr{\mbox{\rm gr}}

\def\e{\mathrm{e}}

\def\m{\mathfrak{m}}

\def\QED{\hfill$\Box$}

\def\Spec{\mbox{\rm Spec}}

\def\xx{{\mathbf x}}
\def\TT{{\mathbf T}}
\def\ff{{\mathbf f}}

\def\g2{{\mathbf g}}

\def\xx{{\mathbf x}}
\def\aa{{\mathbf a}}

\def\ttt{{\mathbf t}}
\def\H{{\mathrm H}}

\def\m{{\mathfrak m}}
\def\p{{\mathfrak p}}

\newcommand{\Proof}{\begin{proof}[\bf Proof]}
\newcommand{\Qed}{\end{proof}}

\newcommand{\depth}{\mathop{\mathrm{depth}}\nolimits}
\def\Ker{\mathrm{Ker}}

\def\e{\mathrm{e}}
\def\m{\mathfrak m}

\def\p{\mathfrak p}

\def\q{\mathfrak q}

\def\H{\mathrm{H}}

\newcommand{\rmj}{\mathrm{j}}

\newcommand{\rmG}{\mathrm{G}}
\newcommand{\rmH}{\mathrm{H}}

\newcommand{\calC}{\mathcal{C}}
\newcommand{\calD}{\mathcal{D}}

\newcommand{\calS}{\mathcal{S}}

\newcommand{\fka}{\mathfrak{a}}

\newcommand{\fkm}{\mathfrak{m}}
\newcommand{\fkn}{\mathfrak{n}}

\newcommand{\fkp}{\mathfrak{p}}
\newcommand{\fkq}{\mathfrak{q}}

\newcommand{\fkM}{\mathfrak{M}}

\def\depth{\mbox{\rm depth }}

\def\Ass{\mathrm{Ass}}

\def\Spec{\mathrm{Spec}}

\begin{document}

\baselineskip=15pt

\title{{Hilbert polynomials of \MakeLowercase{$\mathbf{j}$}-transforms}}

\thanks{AMS 2010 {\em Mathematics Subject Classification}.
Primary 13B10; Secondary 13D02, 13H10,  13H15, 14E05.
The second author was partially supported by the Sabbatical Leave Program at Southern Connecticut State University (Spring 2014).
%, and the last author was partially supported by the NSF
}

\author{Shiro Goto}
\address{Department of Mathematics, School of Science and Technology, Meiji University, 1-1-1 Higashi-mita, Tama-ku, Kawasaki 214-8571, Japan}
\email{goto@math.meiji.ac.jp}

\author{Jooyoun Hong}
\address{Department of Mathematics\\
Southern Connecticut State University\\
501 Crescent Street, New Haven, CT 06515-1533,  U.S.A.}
\email{hongj2@southernct.edu}

\author{Wolmer V. Vasconcelos}
\address{Department of Mathematics \\ Rutgers University \\
110 Frelinghuysen Rd, Piscataway, NJ 08854-8019, U.S.A.}
\email{vasconce@math.rutgers.edu}

\begin{abstract}
% We study $\mathbf{j}$-transforms of finite modules over Noetherian local rings
%that attach to $M$  a graded module $\H_{(\xx)}(M)$ defined
%   via partial systems of parameters $\xx$ of $M$. Despite the generality of the process, in numerous cases, the  $%\mathbf{j}$-transforms  have interesting cohomological properties. We focus on deriving the Hilbert functions
% of $\mathbf{j}$-transforms   and studying the significance of 
%  the vanishing of some of its coefficients.
%\end{abstract}

 We study transformations  of finite modules over Noetherian local rings
that attach to a module $M$  a graded module $\H_{(\xx)}(M)$ defined
   via partial systems of parameters $\xx$ of $M$. Despite the generality of the process, which are called 
   $\mathbf{j}$-transforms, 
   in numerous cases they  have interesting cohomological properties. We focus on deriving the Hilbert functions of $\mathbf{j}$-transforms   and studying the significance of 
  the vanishing of some of its coefficients.
\end{abstract}

\maketitle

%\tableofcontents

\section{Introduction}
Let $(R, \m)$ be a Noetherian local ring and $I$ an ideal of $R$. For a finite $R$-module 
$M$,
our aim is to
study the 
 metrics and homological properties
 associated to the composition of the two  functors 
 \[ M \mapsto \gr_I(M) = \bigoplus_{n\geq 0} I^nM/I^{n+1}M \mapsto \H_I(M) = \H^0_{\m}(\gr_I(M)).\]
 Such transformation $\H_{I}(M)$ could be appropriately called the {\em $\mathbf{j}$-transform} of $M$ relative to $I$.
The study of the $\mathbf{j}$-transform was initiated by
 Achilles and Manaresi \cite{AMa93}
 who made use of the   fact that $\H_{I}(M)=\bigoplus_{k \geq 0}\H_k$
has an associated numerical function
$n \mapsto \psi^{M}_{I}(n)= \sum_{k \le n}\lambda(\H_k)$, where $\lambda(-)$ denotes the length. This function 
 is called the {\em $\mathbf{j}$-function} of $M$ relative to $I$ and is  a broad generalization of the classical Hilbert function -- the case 
 where $I$ is an $\m$-primary ideal.
 Its Hilbert polynomial
\[ \sum_{i = 0}^{r} (-1)^i {\rmj}_i(I;M) {{n+ r-i}\choose{r-i}} \]
 will be  referred  to as the {\em $\mathbf{j}$-polynomial} of $M$ relative to $I$, where $r$ stands for the analytic spread of $I$.
 In general it is very difficult to predict properties of the $\mathbf{j}$-transform $\H_{I}(M)$, beginning
with their Krull dimensions or the {\em $\mathbf{j}$-coefficients} ${\rmj}_i(I;M)$. Nevertheless several  authors have succeeded in applying the construction to extend the full array of classical 
integrality criteria for Rees algebras and modules   
 (\cite{AMa93}, \cite{Ciuperca}, \cite{FM}, \cite{PX12} and \cite{UV08}).

Our goal here is to study a different facet of the $\mathbf{j}$-polynomials. The specific aim is to derive explicit formulas
for the $\mathbf{j}$-coefficients $\rmj_i(I;M)$ in terms of properties of $M$ known a priori and explore the significance of their vanishing. For
that we limit ourselves to ideals generated by partial systems of parameters of $M$ or even special
classes of modules.
Thus we let $\xx= x_1, \ldots, x_r$ be a partial system of parameters of $M$, that
 is $\dim M = r + \dim M/(\xx)M$, and set $I= (\xx)$ and $\rmj_i(\xx;M) = \rmj_i(I;M)$.

A general issue is what the values of $\rmj_1(\xx;M)$ say about $M$ itself.
In \cite{chern3}, \cite{chern5},  \cite{chern2} and \cite{chern}, the authors and colleagues  studied the values of a special class of these coefficients.
For a Noetherian local ring $R$ and a finite $R$-module $M$, we considered the Hilbert coefficients
$\e_i(\xx;M)$ associated to filtrations defined by a system  $\xx$ of parameters of $M$, more precisely to the Hilbert
functions
\[ n \mapsto \lambda(M/(\xx)^{n+1}M)\]
and made use of the values of $\rmj_1(\xx;M)$
as the means to detect various properties of $M$ (e.g., Cohen-Macaulay, Buchsbaum, finite local cohomology, etc.). Here we seek to extend these probes to cases when $r< \dim M$.  The significant distinction between $\gr_{(\xx)}(M)$ and $\H_{\m}^0(\gr_{(\xx)}(M))$ is that when $r < \dim M$, the latter may not
be homogeneous and therefore the vanishing of some of its Hilbert coefficients does not place them
entirely in the context of  \cite{chern3},  \cite{chern2} and \cite{chern}.

We illustrate these issues with a series of  questions. Let $R$ be a Noetherian local ring and $M$ a finite $R$-module.  Let
$I$ be an ideal generated by a partial system of 
parameters $\xx= x_1, \ldots, x_r $  of $M$ and let $\rmG=\gr_{I}(R)$ be the associated graded ring.   Then the $\rmG$-module
$\H_{I}(M)= \H_{\m}^0(\gr_{I}(M))$ has dimension at most $r$ (Proposition~\ref{lemma1}).
We list some questions similar to those  raised in \cite{chern} for a  full system of parameters:

\begin{itemize}

\item[{\rm (i)}] What are the possible values of $\dim_{\rmG} \H_{I}(M) $? Note that $\H_{I}(M)=(0)$ may happen or $\H_{I}(M) \neq (0)$ but of dimension zero.

\item[{\rm (ii)}] What is the signature of $\rmj_1(\xx; M)$?
 If $r= \dim_{\rmG} \H_{I}(M)$, is $\rmj_1(\xx; M) \leq 0$? The answer is affirmative if $\H_{I}(M)$ is generated in 
 degree $0$, because $\rmj_1(\xx; M) = \e_1(\ff, \H_{I}(M))$ and
these coefficients are always non-positive according to \cite[Theorem 3.6]{MSV} and \cite[Corollary 2.3]{chern5}, where $\ff = f_1, \ldots, f_r $ denotes the initial forms of $x_i's$ relative to $I$.

\item[{\rm (iii)}] If $M$ is unmixed,   $r=\dim_{\rmG} \H_{I}(M)$, and  $\rmj_1(\xx; M)=0$, then is $\xx$  an $M$-regular sequence?
The answer is obviously no. What additional restriction is required?
\end{itemize}

The questions (ii) and (iii) were dealt with in \cite{chern3}, \cite{chern5}, \cite{chern2}  and \cite{chern} for $r=\dim M$, but we do not know much in the other cases. 
We shall focus on a special kind of partial systems of parameters which will be shown to be ubiquitous.
For   a finite $R$-module $M$, we call a partial system $\xx= x_1, \ldots, x_r $ of parameters of $M$ {\em amenable} if the Koszul homology module  $\H_1(\xx;M)$ has finite length (and hence $\H_i(\xx;M)$ have finite
length for all $i\geq 1$).

Let us state one of our main questions.

\begin{Conjecture}\label{j1is0}{\rm Let $(R, \m)$ be a Noetherian local ring and 
$M$ a finite unmixed  $R$-module. Let $\xx=x_1, \ldots, x_r $
be a partial system of parameters of $M$. Suppose that $\xx$ is an amenable $d$-sequence relative to $M$, that  $\dim \H^0_{\m}(\gr_{(\xx)}(M))=r$,  and that $\rmj_1(\xx;M)= 0$.  Then $\xx$ is a regular sequence on $M$. 
}\end{Conjecture}

The unmixedness hypothesis serves various purposes that begin from avoiding trivial counterexamples to 
the enabling of several constructions for cohomological calculations.
 There are two versions
of the conjecture, the {\bf strong} version as stated and a {\bf weak} one. The latter assumes that
$\rmj_1(\xx;M) = 0$ for {\bf all} partial systems of parameters of $r$ elements.

In outline, our approach to describe some of the $\mathbf{j}$-transform $\H_{\m}^0(\gr_{(\xx)}(M))$ is the following. First we show how to find an amenable partial system of parameters of $M$ that is a $d$-sequence relative to $M$ (Proposition~\ref{ubiquityth}). 
 Once provided with such a partial system  $\xx$ of parameters, $\H_{\m}^0(\gr_{(\xx)}(M))$ occurs as the homology of the complex obtained by applying the section functor $\H_{\m}^0(-)$ to
  the approximation complex $\mathcal{M}(\xx;M)$ (\cite{HSV3II}):
  \begin{center}
{\small $ 0 \rar  \H_{\m}^0(\H_r(\xx;M)) \otimes S[-r] \rar \cdots \rar  \H_{\m}^0(\H_1(\xx;M))\otimes S[-1] \rar \H_{\m}^0(\H_0(\xx;M)) \otimes
 S \rar \H_{\m}^0(\gr_{\xx}(M)) \rar 0,$}
 \end{center}
where $S = R[\TT_1, \ldots, \TT_r]$ is the polynomial ring (see Theorem~\ref{jdseqcx}). When we add to this the assumption that $\H_1(\xx;M)$ has finite length, we derive the Hilbert series $[\![\rmH]\!]$ of $\H= \H_{\m}^0(\gr_{(\xx)}(M))$ in terms of the Koszul homology modules: If
$h_i = \lambda(\H_i(\xx;M))$, $i>0$, and $h_0 = \lambda(\H_{\m}^0(\H_0(\xx;M))$, then
\[ [\![\H]\!]= \frac{\sum_{i=0}^r (-1)^i h_i \ttt^i}{(1-\ttt)^{r}}, \]
where $\ttt$ denotes the indeterminate. 
Several properties of $\H$ can be read right away from the complex, such as the vanishing of $\H$ and  the Cohen-Macaulayness of $\H$ (Corollary~\ref{vanishingofH}). Both of which require 
that $\xx$ be a regular sequence on $M$.  Besides the expression of the Hilbert function of $\H$ in terms of the 
values of $ \lambda( \H_i(\xx;M))$, 
 Corollary~\ref{jdseqCor} interprets several properties of the relationship between $\xx$ and $M$.
 For instance, if $\dim \H = r$, then $\rmj_1(\xx;M)\leq 0$, the exact behavior of the case when
 $\xx$ is a full system of parameters. However, it is not enough  to determine the Krull dimension of $\H$ much less to resolve our conjecture
 on whether the vanishing of $\rmj_1(\xx;M)$  means that $\xx$ is a regular sequence on $M$. This would require that $\H$ be unmixed along with $M$.
In the remaining of the section we establish the conjecture in several special cases
 that places restrictions on the projective dimension of $M$ or the $\xx$-depth of $M$.

In the following sections we develop formulas for the $\mathbf{j}$-coefficients based on the finiteness of
of the local cohomology associated to certain partial systems of parameters.  In section 4,  we derive (see
 Theorem~\ref{19}) the Hilbert function of $\H$ as
 \[\psi^{M}_{(\xx)}(n) = \sum_{i=0}^{r}k_{i}(\xx ; M) \binom{n+r-i}{r-i} \quad
 \mbox{\rm  for all $n \ge 0$}\] 
with the  $k_{i}(\xx; M)$ being some positive linear combinations of the lengths of local cohomology modules. They
are therefore very convenient to study relationships between depth conditions and the vanishing of the $\rmj_i(\xx; M)$.

The focus in Section 5 is on the boundedness of {\it all} $\mathbf{j}$-coefficients $\rmj_i(\xx;M)$ for {\it all} amenable partial
systems of parameters which are $d$-sequences relative to the module $M$, a 
problem treated in \cite{chern3} and \cite{chern5} for full systems of parameters. The corresponding result  (Theorem~\ref{Finjplus}) in this more delicate
setting
 requires modules admitting one strong $d$-sequence, a property that holds for all
 homomorphic images of Gorenstein rings.
Finally, in Section 6 we give structure theorems for the $\mathbf{j}$-transforms of Buchsbaum and sequentially 
Cohen-Macaulay modules. They come with a detailed description of their Hilbert functions.

\section{Amenable partial systems of parameters} 

In this section we discuss a class of partial systems of parameters that provides a gateway to the calculation
of numerous $\mathbf{j}$-polynomials because of their flexibility in the calculation of local cohomology.

\begin{Definition}{\rm 
  Let $R$ be a Noetherian local ring and let
$M$ be a finite $R$-module of dimension $d>0$.
A partial system $\xx=x_1, \dots, x_r $ ($0 \leq r \leq d$) of parameters of $M$ is said to be {\em amenable} to $M$, if  
the Koszul homology modules $\H_i(\xx;M)$ have finite length for all $i >0$. 
}\end{Definition}

A simple example is that of a Noetherian local ring $R$ of dimension $2$, with $x$ satisfying
$0:x = 0:x^2$ and $x$ not contained in any associated prime $\p$ of $R$ of codimension at most one.
Similarly if $M$ is Cohen-Macaulay on the punctured support of $M$, then any partial system of parameters  is amenable. Under mild additional conditions, these are the modules with {\it finite local cohomology} (FLC), that is
the cohomology modules $\H_{\m}^i(M)$, $i< \dim M$, have finite length. 
 
Given the rigidity of Koszul complexes, the condition is 	equivalent to saying that $\H_1(\xx;M)$ has
finite length. 
These sequences are also called  {\em filter regular} systems of parameters (see
\cite{Schenzel}), but the main source of amenable partial systems of parameters is found in a property of $d$-sequences. 
Let us recall briefly this notion which we pair  to  a related type of sequence.
 They are
  extensions of regular sequences,  the notion of  {\em
  $d$-sequence} \index{d-sequence} invented by Huneke (\cite{Hu1a}),
  and of a {\em proper sequence} (\cite{HSV3II}). They play natural
  roles in the theory of the {\em approximation complexes}
  (see \cite{HSV3II}).
We now collect appropriate properties of these
sequences.
 
  \begin{Definition} {\rm Let $R$ be a commutative ring and $M$ an $R$-module. A sequence $\xx=x_1,\ldots, x_n$ of $R$-elements is called a {\em $d$-sequence} relative to $M$,
if \[ (x_1, \ldots, x_{i-1})M:_M x_{i}x_j = (x_1, \ldots, x_{i-1})M:_M x_j  \ \ \mbox{for all} \  \ 1 \le i \le j \le n.\]
A sequence $\xx=x_1,\ldots, x_n$ of $R$-elements is called a {\em strong $d$-sequence} relative to $M$, 
if for every sequence $r_1, \ldots, r_n$ of positive integers the
sequence $x_1^{r_1}, \ldots, x_n^{r_n} $ is a $d$-sequence relative to $M$.  We say that $\xx=x_1,\ldots, x_n$ is a {\it $d^+$-sequence} relative to $M$,  if it is a strong $d$-sequence relative to $M$ in any order.
}\end{Definition}

\begin{Definition} {\rm Let $R$ be a commutative ring and $M$ an $R$-module.  A sequence $\xx= x_1,  \ldots, x_n$  of
$R$-elements is a  {\em proper sequence} relative to $M$, if
\[ x_{i+1}\H_j(x_1, \ldots, x_i ; M)=0 \ \ 
\mbox{for all} \ \ 1 \le i < n\ \ \mbox{and}\ \  j>0,\]
where $\H_j(x_1, \ldots, x_i ;M)$ is the Koszul homology of $M$ associated to 
$x_1, \ldots, x_i $. 
}\end{Definition}

\begin{Proposition}\label{amproperseq}
Let $R$ be a Noetherian ring and 
   $\xx = x_1, \ldots, x_n$  a sequence in $R$.
\begin{enumerate}
\item[{\rm (1)}] If $\xx$ is a d-sequence relative to $M$, then $\xx$ is a proper sequence relative to $M$.

\item[{\rm (2)}] If $\xx$ is a proper sequence relative to $M$, then \[
x_{k}\H_j(x_1,  \ldots, x_i ; M)=0\ \   \mbox{for all} \ \ 1 \le i <n,  \ j>0,  \  \mbox{and}\ \ k>i.\]

 \item[{\rm (3)}] Suppose that $R$ is a local
ring and $M$ a finite $R$-module of dimension $n>0$. If $\xx$ is  a system of parameters of $M$ that is also a proper sequence relative to $M$, then $\H_j(x_1,  \ldots, x_i ;M)$ has finite length for all $1 \le i \le n$ and $j > 0$.
\end{enumerate}
\end{Proposition}

We make use of Proposition \ref{amproperseq} (3) to build amenable partial systems of parameters of $M$. 

\medskip

\subsubsection*{Ubiquity}

 The following assertion is based on \cite[Proposition 2.7]{HSV3II}. That construction applied only to
$M=R$ and had a missing sentence. We use the proof to build a system of parameters that is a $d$-sequence whose first $r$
elements are given.

\begin{Proposition} \label{ubiquityth}  Let $(R, \m)$ be a Noetherian local ring and $M$ a finite $R$-module of dimension $d>0$. Let $I$ be an $\fkm$-primary ideal of $R$. 
Then there exists a system of parameters of $M$ which is contained in $I$ and a $d$-sequence relative to $M$.
\end{Proposition}

\begin{proof} For an $R$-module $C$ we set $z(C)= \bigcup_{P \in \Ass_{R}C \setminus\{\fkm\}}P$. Let $y_{1} \in I \setminus z(M)$ and pick $m$ large enough so that $(0):_{M} y_1^m = (0):_{M} y_1^{m+1}$. Let $x_{1}=y_{1}^{m}$ and $J_{1}= (0)  :_{R} ((0):_{M}x_{1})$. For $1 \le s \le d-1$ we define the following recursively. Let 
\[ y_{s+1} \in (I \cap J_{1} \cap \ldots \cap J_{s}) \setminus (z(M) \cup z(M/x_1M) \cup \ldots \cup z(M/(x_1, \ldots, x_s)M) ),\]
where $J_{i}= (x_{1}, \ldots, x_{i-1})M :_{R} ( (x_{1}, \ldots, x_{i-1})M :_{M} x_{i} )$  for each $1 \le i \le s$.
Pick $m$ large enough so that 
\[ (x_{1}, \ldots, x_{i})M :_{M} y_{s+1}^{m} = (x_{1}, \ldots, x_{i})M :_{M} y_{s+1}^{m+1} \quad \mbox{\rm for all}\;\; 0 \le i \le  s.\]
We set $x_{s+1}= y_{s+1}^{m}$. 
Then since the elements $x_{1}, \ldots, x_{d}$ as chosen generate an ideal of height $d$, it is enough to prove that every subsequence $x_{1}, \ldots, x_{s}$ is a $d$-sequence relative to $M$ for all $1 \le s \le d$.  We use induction on $s$. Let $s=1$. Then by definition, $(0):_{M} x_{1} = (0):_{M} x_{1}^{2}$, which means that $x_{1}$ is a $d$-sequence relative to $M$. Suppose that $x_{1}, \ldots, x_{s}$ is a $d$-sequence relative to $M$. In order to prove that $x_{1}, \ldots, x_{s}, x_{s+1}$ is a $d$-sequence relative to $M$, it is enough to show   that for all $1 \le i \le s+1$
\[ (x_{1}, \ldots, x_{i-1})M :_{M} x_{i} x_{s+1} \subseteq (x_{1}, \ldots, x_{i-1})M :_{M} x_{s+1}.\]
Let $u \in M$ such that $(x_{i}x_{s+1}) u \in (x_{1}, \ldots, x_{i-1})M$. Then $ x_{s+1} u \in ( (x_{1}, \ldots, x_{i-1})M :_{M} x_{i})$. 
Since $x_{s+1} \in J_{i}$ for all $1 \le i \le  s$, we get
\[ x_{s+1}( x_{s+1}u) \in (x_{1}, \ldots, x_{i-1})M.\]
Hence
\[ u \in (x_{1}, \ldots, x_{i-1})M :_{M} x_{s+1}^{2}  =  (x_{1}, \ldots, x_{i-1})M :_{M} x_{s+1} \] as asserted.  \end{proof}

\begin{Corollary} Let $R$ be a Noetherian local ring and $M$ a finite $R$-module of dimension $d>0$. Then for each $0 \le r \leq d$ there exists an amenable partial system of parameters of $M$ of length $r$ that is a $d$-sequence relative to $M$.
\end{Corollary}

\begin{proof} By Proposition \ref{ubiquityth} there exists a system $x_{1}, \ldots, x_{d}$ of parameters  of $M$ that is a $d$-sequence relative to $M$ and by Proposition \ref{amproperseq} (3)
 the partial system $x_{1}, \ldots, x_{r}$ is  an amenable $d$-sequence relative to $M$.
\end{proof}

\begin{Corollary}
Let $(R, \m)$ be a Noetherian local ring and $M$ a finite $R$-module of dimension $d>0$.
Let $\xx= x_{1}, \ldots, x_{r} $ be an amenable partial system of parameters of $M$ that is  a $d$-sequence relative to $M$. Then the sequence $\xx$ can be extended to a full system of parameters of $M$ that is a $d$-sequence 
relative to $M$.
\end{Corollary}

\begin{proof} Apply the construction of Proposition \ref{ubiquityth} to $I=\fkm$, choosing
$\xx$ for the first $r$ elements of the $d$-sequence.
\end{proof}

\begin{Question}{\rm
Let $R$ be a Noetherian local ring and $M$ a finite  $R$-module. Given an amenable partial system $\xx= x_{1}, \ldots, x_{r} $  of parameters of $M$, is there 
 an amenable partial system $y_{1}, \ldots, y_{r} $  of parameters of $M$ in the ideal $(\xx)$ that is a $d$-sequence relative to $M$?
}\end{Question}

\medskip

\subsubsection*{Signature of \  $\mathbf{j}_{1}(\xx;M)$}  
For an $R$-submodule $N$ of $M$, we let ${ N:_M\left<\m\right> = \bigcup_{\ell > 0}[N:_M\fkm^\ell]}$. Hence $\rmH_\fkm^0(M/N) = [N:_M \left<\m\right>]/N$.  

\begin{Theorem}\label{14} Let $(R,\m)$ be a Noetherian local ring and $M$ a finite $R$-module of dimension $d > 0$. Let $I=(\xx)=(x_1, \ldots, x_r)$ be an ideal of $R$ generated by an amenable partial system of parameters of $M$ that is a $d$-sequence relative to $M$, where $0 \leq r \leq d$.
Let $\rmG = \gr_I(R)$. Then the following assertions hold true.
\begin{enumerate}
\item[{\rm (1)}] $I^nM \cap [I^{n+1}M :_M \left<\m\right>] = I^n[IM:_M \left<\m\right>]$ for all $n \geq 0$.
\item[{\rm (2)}]  The $\mathbf{j}$-transform $\rmH_{I}(M)$ is generated by its homogeneous component of degree $0$, i.e., 
\[ \rmH_{I} (M) = \rmG{\cdot}[\rmH_{I} (M)]_0.\]
\end{enumerate}
\end{Theorem}

\begin{proof} (1)
We have only to show $I^nM \cap [I^{n+1}M :_M  \left<\m\right> ] \subseteq I^n[IM:_M \left<\m\right> ]$. By induction, we may assume that $n, r > 0$ and the assertion holds true for $n -1$ and $r -1$. We set $\overline{M} = M/x_1M$ and consider the partial system $x_2, \ldots, x_r$ of parameters of $\overline{M}$. Let $ f \in I^{n+1}M :_M \left<\m\right> $ and let $\overline{f}$ denote the image of $f$ in $\overline{M}$. Then $\overline{f} \in I^n\overline{M} \cap [I^{n+1}\overline{M} :_{\overline{M}}\left<\m\right> ]$ and hence $f \in I^{n}[IM:_M  \left<\m\right>] + x_1M$ by the hypothesis on $r$. Therefore without loss of generality, we may assume that $f \in x_1M \cap I^nM$. Hence, because $x_1M \cap I^nM = x_1I^{n-1}M$ (as $\xx$ is a $d$-sequence relative to $M$; see \cite[Proposition 2.2]{Hu1}), we get  $f = x_1g$ for some $g \in I^{n-1}M$. Then for  $\ell \gg 0$, $x_1(\m^{\ell}g) \subseteq I^{n+1}M \cap x_1M ~= x_1I^nM$. Let $a \in \m^\ell$ and write $x_1 (a g) = x_1 h$ with $h \in I^nM$.  Then 
$a g - h \in [(0):_M x_1] = \H_{r}(\xx ; M)$. Since $\xx$ is amenable, by using the hypothesis of induction, we get
\[ g \in I^{n-1}M \cap [ I^{n}M :_{M} <\m> ] = I^{n-1}[ IM :_{M} <\m> ].\]
This means that $f = x_{1} g \in I^{n-1} [ IM :_{M} <\m> ]$ as asserted.

(2) It follows from $\rmH^0_\m(I^nM/I^{n+1}M) = \left\{I^n[IM:_M \m ]\right\}/I^{n+1}M$ for all $n \ge 0$ by (1).
\end{proof}

\begin{Corollary}\label{j1sig}
Let $(R,\m)$ be a Noetherian local ring and let $M$ be a finite $R$-module of dimension $d > 0$. Let $\xx = x_1, \ldots, x_r$ be an  amenable partial system of parameters of $M$ that is a $d$-sequence relative to $M$. Suppose that $\dim \rmH_{(\xx)}(M)=r$. Then $\rmj_{1}(\xx; M) \leq 0$.
\end{Corollary}

\begin{proof} Let $\H= \rmH_{(\xx)}(M)$ and let $\ff = f_1, \ldots, f_r$ denote the initial forms of $x_i's$ relative to $I$ in $\rmG= \gr_I(R)$. Then $\gr_{(\xx^*)}(\H) = \H$ by (\ref{14}), so that
$\rmj_1(\xx; M) = \e_1(\xx^*, \H) \leq 0$  by \cite[Theorem 3.6]{MSV} and \cite[Corollary 2.3]{chern5}. 
\end{proof}

To get $\rmH_I(M) = \rmG {\cdot}[\rmH_I(M)]_0$, one really needs the assumption that $\xx$ is a $d$-sequence relative to $M$ which is amenable. Let us explore an example. See Example \ref{6.12} to agree with the independence of the $d$-sequence property and amenability.

\begin{Example}
{\rm
Let $(S,\fkn)$ be a regular local ring of dimension $q+3~(q \ge 0)$. We write
$\fkn = (X,Y,Z, W_1, \ldots, W_q)$. Let $\fka = (X) \cap (Y^2, Z) \cap (X^2, Y,Z^2)=(X^2Z, XY^2, XYZ, XZ^2)$ and set $R = S/\fka$. Then $\dim R = q+2$. Let $x,y, z$ and $w_i$ denote, respectively, the images of $X,Y, Z$, and $W_i$ in $R$. We consider the partial system $z, w_1, \ldots, w_q$  of parameters of $R$. Let $I = (z, w_1, \ldots,w_q)$ and $\mathrm{G} = \operatorname{gr}_I(R)$. We set $\rmH = \rmH_\fkm^0(\rmG)$ and let $\fkm = \fkn / \fka$ denote the maximal ideal of $R$. We then have the following.

\begin{Lemma}\label{applemma} The following assertions hold true.
\begin{enumerate}
\item[$(1)$] $R/I$ is a Cohen-Macaulay ring and $0 \ne \overline{xz} \in \rmH_\fkm^0(I/I^2)$, where $\overline{xz}$ denotes the image of $xz$ in $I/I^2$. 
\item[$(2)$] $w_1, \ldots, w_q$ is a super-regular sequence with respect to $I$, that is the initial forms of $w_i$'s constitute a regular sequence in $\mathrm{G}$. 
\item[$(3)$] If $q = 0$, then $\rmH = \rmH_1$.
\item[$(4)$] $\rmH = \rmG{\cdot}\rmH_1$ and $\lambda(\rmH_1) = 1$. 
\end{enumerate}
\end{Lemma}

\begin{proof} 
(1) The first assertion is clear, as $R/I = S/[(Z, W_1, \ldots, W\q) + (XY^2)]$. We set $K = (W_1, \ldots, W_q)$ and $J = (Z, W_1, \ldots,W_q)$.  Then $I/I^2 \cong J/[J^2 + (\fka \cap J)]$. Since
$J^2 + (\fka \cap J) = J^2 + (X^2Z, XYZ, XZ^2) + (XYW_i \mid 1 \le i \le q)$, we have $XZ \not\in J^2 + (\fka \cap J)$, while $\fkn{\cdot}XZ \subseteq J^2 + [\fka \cap J]$. 

(2) We may assume that $q > 0$. As $w_1$ is $R$-regular and $R/(w_1) = S/[(W_1) + \fka]$, we have by induction on $q$ that $w_1, \ldots, w_q$ is an $R$-regular sequence. We must show that $(w_1, \ldots, w_q) \cap I^n = (w_1, \ldots, w_q)I^{n-1}$ for all $n \ge 2$. Notice that
$$(K + \fka)\cap (J^n + \fka) = \fka + \left(K \cap [KJ^{n-1} + (Z^n) +\fka]\right)
= (\fka + KJ^{n-1})+K\cap [(Z^n) + \fka].$$ 
We then have $(K + \fka)\cap (J^n + \fka) = \fka + KJ^{n-1}$, since $K\cap [(Z^n) + \fka]= K{\cdot}[(Z^n) + \fka]$.

(3), (4)  We have $\fka :_S Z = X(X,Y,Z)$ and $\fka :_S Z^n = (X)$ for all $n \ge 2$, whence $(0):_R z = x(x,y,z)$ and $(0):_R z^n = (x)$. Therefore $(z^n)/(z^{n+1}) \cong R/(x,z) = S/(X,Z)$, so that $\rmH_\fkm^0((z^n)/(z^{n+1})) = (0)$ for $n \ge 2$, which shows that if $q = 0$, then $\rmH = \rmH_1$ with $\lambda(\rmH_1) = 1$ (remember that $(z)/(z^2) \cong S/(X\fkn + (Z))$, if $q = 0$). Suppose now that $q > 0$ and that the assertion holds true for $q - 1$. Let $\overline{R} = R/(w_1)$ and set $\rmG (\overline{R}) = \gr_{I\overline{R}} (\overline{R})$. We consider the exact sequence 
$$0 \to \rmG[-1] \overset{f}{\to} \rmG \to \rmG (\overline{R}) \to 0$$
of graded $\rmG$-modules, where $f$ denotes the initial form of $w_1$. Then applying $\rmH_\fkm^0(-)$, we get the exact sequence
$$(\sharp_1) \ \ \ 0 \to \rmH[-1] \overset{f}{\to} \rmH \to \rmH (\overline{R})$$
of graded $\rmG$-modules, where $\rmH (\overline{R}) = \rmH_\fkm^0(\rmG(\overline{R}))$. Then because $\rmH_0 = (0)$, we have an embedding $0 \to \rmH_1 \to [\rmH(\overline{R})]_1$, which must be surjective, as $\rmH_1 \ne (0)$ by assertion (1) and $\lambda([\rmH(\overline{R})]_1)= 1$ by the hypothesis of induction. Therefore, because $\rmH(\overline{R})=\rmG(\overline{R}){\cdot}[\rmH(\overline{R})]_1$ by the hypothesis of induction, exact sequence $(\sharp_1)$ is actually a short exact sequence
$$(\sharp_2)\ \ \ 0 \to \rmH[-1] \overset{f}{\to} \rmH \to \rmH(\overline{R}) \to 0,$$
so that $\rmH = \rmG{\cdot}\rmH_1 + f\rmH$, whence $\rmH = \rmG{\cdot}\rmH_1$ by Nakayama's lemma. The equality $\lambda(\rmH_1) = 1$ is now clear, since $\rmH_1 \cong [\rmH(\overline{R})]_1$.
\end{proof}

We set $\overline{G} = \rmG/\fkm \rmG$. Notice that $\rmH$ is a graded $\overline{\rmG}$-module, as $\fkm \rmH = (0)$ (Lemma \ref{applemma} (4)). Let $h=xz + I^2 \in \rmH_1$ and let $F$ denote the image of $z + I^2$ in $\overline{\rmG}$. With this notation our conclusion is the following.

}
\end{Example}

\begin{Proposition}
$[\![\rmH]\!] = \displaystyle\frac{\mathbf{t}}{(1 - \mathbf{t})^q}$ and $(\overline{\rmG}/F\overline{\rmG})[-1] \cong \rmH$ as a graded $\overline{\rmG}$-module.
\end{Proposition}

\begin{proof} If $q = 0$, then $\rmH = \rmH_1$ and $\lambda(\rmH_1) = 1$, so that  the assertions are certainly true. Suppose that $q > 0$ and that our assertion holds true for $q-1$. Then thanks to exact sequence $(\sharp_2)$ in the proof of Lemma \ref{applemma},  $[\![\rmH]\!] =\frac{[\![\rmH(\overline{R})]\!] }{1-\mathbf{t}}$, whence $[\![\rmH]\!] = \frac{\mathbf{t}}{(1 - \mathbf{t})^q}$, as $[\![\rmH(\overline{R})]\!] = \frac{\mathbf{t}}{(1 - \mathbf{t})^{q-1}}$  by the hypothesis of induction.  Let $\varphi : \overline{G}[-1] \to \rmH$ be the $\overline{G}$-linear map defined by $\varphi (1) = 
h$. Then the homomorphism $\varphi$ is surjective, since $\rmH_1 = Rh$ (Lemma \ref{applemma} (1), (4)). Therefore, as $F\rmH = (0)$ (remember that $xz^2 =0$ in $R$), we get an epimorphism $$(\overline{\rmG}/F\overline{\rmG})[-1] \to \rmH \to 0$$ of graded $\overline{G}$-modules, which forces $(\overline{\rmG}/F\overline{\rmG})[-1] \cong \rmH$, because the Hilbert series of both sides are the same.
\end{proof}

Here we notice that $(0):_R z = x(x,y,z)$ and $(0):_Rz^2 = (x)$, whence $z, w_1, \ldots, w_q$ is not a $d$-sequence relative to $R$. We also see that $z, w_1, \ldots, w_q$ is not amenable, as $zR_\fkp = (0)$ for the prime ideal $\fkp = (y,z, w_1, \ldots, w_q)$ of $R$.

\medskip

\subsubsection*{Amenable sequences and regular sequences}
 
The following  elementary observation shows the strength of a condition to decide which partial system of parameters is already a regular
sequence.

\begin{Proposition} \label{amenablearereg}
Let $(R, \m)$ be a Cohen-Macaulay local ring of dimension $n$ and $M$ a finite $R$-module
of projective dimension $s$. Let $\xx=x_1, \ldots, x_r$, $1 \le r\leq n-s$, be a partial system of parameters for $R$ and $M$.
If $\xx$ is amenable to $M$, then  it is  an $M$--regular sequence.
\end{Proposition}

\begin{proof}
Let $K$ be the module of syzygies of $M$:
\[ 0 \rar K \lar F \lar M \rar 0.\]
Since $\H_i(\xx;F)=(0)$, $i\geq 1$, we have $\H_i(\xx;K) = \H_{i+1}(\xx;M)$ for $i\geq 1.$
This means that $\xx$ is amenable for $K$, and by induction on the projective dimensions, $\xx$ is a regular
sequence on $K$.
Since $r < \depth  K$, $\H_{\m}^0(K/(\xx)K) = (0)$. From the
the exact sequence
\[ 0 \rar \H_1(\xx;M) \lar K/(\xx)K \lar F_0/(\xx)F_0 \lar M/(\xx)M \rar 0\]
 we get $\H_{\m}^0(\H_1(\xx;M)) = (0)$, which by assumption means $\H_1(\xx;M)=(0)$.
\end{proof}

\medskip

\subsubsection*{Amenable sequences and finite local cohomology}
Let $R$ be a Noetherian local ring and $M$ a finite $R$-module of dimension $d > 0$. 
The authors in \cite{chern7} attached to a system $\xx=x_1, \ldots, x_d$ of parameters of $M$ (which is naturally amenable to $M$)
 the quantity 
\[ \chi_1(\xx;M) = \sum_{i= 1}^{d} (-1)^{i-1} \lambda( \H_{i}(\xx;M) ),\]
which is called the {\it partial} Euler characteristic.  If $\xx= x_1, \ldots, x_r $~$(0 < r \le d)$ is the initial subsequence of
a full system $ x_1, \ldots, x_d $ of parameters that is a $d$-sequence relative to $R$, they lead to explicit formulas for {\it some} of the standard Hilbert coefficients $\e_i(\xx; R)$ (\cite[Theorem 3.7]{chern7}) and their global bounds.  In similar situations, we will obtain bounds for all the $\rmj_i(\xx;M)$ (see Section 5).

\section{Koszul homology: Explicit formulas for $\mathbf{j}$-coefficients}

In this section 
we derive a general formula for the $\mathbf{j}$-coefficients of a general class of modules.
Let us begin by introducing some notation. Throughout this section, 
let $(R,\m)$ be a Noetherian local ring and $M$  a finite $R$-module  of dimension $d > 0$. Let $\xx=x_1,  \ldots, x_r $ be a partial system of parameters of $M$, where $0 \leq r \leq d$. We denote by $(\xx)$ the ideal generated by $\xx$. 
Let $\rmG = \mathrm{gr}_{(\xx)}(R)$, $\rmG (M) = \mathrm{gr}_{(\xx)}(M)$, and  $\H = \rmH_\m^0(\rmG (M))$. Finally, let $\H_i(\xx;M)$ denote the Koszul homology modules of $M$ relative to $\xx$.
Let us begin by highlighting  some of its general properties.

\begin{Proposition}\label{lemma1} We have the following.
\begin{enumerate}
\item[$(1)$] $\m \rmG \in \Spec \rmG$ and $\Ass_{\rmG} \H = \{P \in \Ass_{\rmG} \rmG(M) \mid \m \rmG \subseteq P\}$. In particular,  $\dim \rmG/P \le r$ for all $P \in \Ass_{\rmG} \H$. 
\item[$(2)$] If $\H \ne (0)$, then $\dim_{\rmG} \H \le r$, where the equality holds if and only if $\m \rmG \in \Ass_{\rmG} \H$. 
\item[$(3)$] If $r < d$ and $\dim \rmG/P = d$ for all $P \in \Ass_{\rmG} \rmG (M)$, then $\H = (0)$.
\end{enumerate}
\end{Proposition}

\begin{proof}
(1) Since $\rmG/ \m \rmG \cong (R/\m)[\TT_1, \TT_2, \ldots, \TT_r]$ 
is the polynomial ring, we have $\m \rmG \in \Spec \rmG$.
 Let $P \in \Ass_{\rmG} \H$. Then $P \in \Ass_{\rmG} \rmG (M)$ and $P = (0):_{\rmG} \varphi$ for some $\varphi \in \H \setminus (0)$,
  whence $\m \rmG \subseteq P$, because $\m^\ell \varphi = (0)$ for $\ell \gg 0$. Conversely, assume that $P \in \Ass_{\rmG} \rmG (M)$ with $\m \rmG \subseteq P$. We write $P = (0):_{\rmG} \varphi$ for some $\varphi \in \rmG (M) \setminus (0)$. Then $\m \varphi = (0)$ since $\m \rmG \subseteq P$, so that $P \in \Ass_{\rmG} \rmH$. Of course, $\dim \rmG/P \le r$ if $P \in \Ass_{\rmG} \H$, because  for all $P \in \Ass_{\rmG} \H$, $\rmG/P$ is a homomorphic image of $\rmG/\m \rmG$ and $\dim \rmG/ \m \rmG = r$.
Assertions (2), (3) now follow from (1).
\end{proof}

One of  our main techniques to determine $\mathbf{j}$-polynomials is the following, where unadorned tensor products are over $R$.

\begin{Theorem}\label{jdseqcx} 
%Let $(R, \m)$ be a Noetherian local ring and $M$ a finite $R$-module of dimension $d > 0$. 
Let $\xx=x_1, \ldots, x_r$~$(0 < r \le d)$ be a partial system of parameters of $M$ that is a $d$-sequence relative to $M$. Let $S=R[\TT_1, \ldots, \TT_r]$ be the polynomial ring. Then  there exists a complex of graded $S$-modules:
 \begin{center}
{\small $ 0 \rar  \H_{\m}^0(\H_r(\xx;M)) \otimes S[-r] \rar \cdots \rar  \H_{\m}^0(\H_1(\xx;M))\otimes S[-1] \rar \H_{\m}^0(\H_0(\xx;M)) \otimes
 S \rar \H_{\m}^0(\gr_{(\xx)}(M)) \rar 0.$}
 \end{center}
Furthermore, if $\xx$ is amenable to $M$, then the following complex is acyclic:
 \begin{center}
{\small $ 0 \rar \H_r(\xx;M) \otimes S[-r] \rar \cdots \rar  \H_1(\xx;M)\otimes S[-1] \rar \H_{\m}^0(\H_0(\xx;M)) \otimes
 S \rar \H_{\m}^0(\gr_{(\xx)}(M)) \rar 0.$}
 \end{center}
\end{Theorem}

\begin{proof}
We refer to \cite{HSV3II} for details about the approximation complexes $\mathcal{M}(\xx;M)$ used here:
\[ 0 \rar \H_r(\xx;M)\otimes S[-r] \rar \cdots \rar \H_1(\xx;M)\otimes S[-1] \rar \H_0(\xx;M) \otimes
 S \rar \gr_{(\xx)}(M) \rar 0.\]  This 
complex  is an acyclic complex of  graded $S$-modules, %(unadorned tensor products are over $R$)  
because $\xx$ is a $d$-sequence relative to $M$(\cite[Theorem 4.1]{HSV3II}). Our complex arises from applying the functor $\H_{\m}^0(-)$ to $\mathcal{M}(\xx;M)$. 
 Now assume that $\xx$ is amenable to $M$, so that $ \H_i(\xx;M) =
 \H_{\m}^0(\H_i(\xx;M))$ for all $i\geq 1$. 
We note that the image $L$ of $\H_1(\xx;M)\otimes S[-1]$ in $\H_0(\xx;M) \otimes S$ has finite length and therefore
$\H_{\m}^1(L)=0$. Since for all $i\ge 1$ $\H_i(\xx;M) \otimes_{R} S[-i]$ is  supported in the set $\{\m\}$, we obtain the 
exact complex $\H_{\m}^0(\mathcal{M}(\xx;M))$:
\[ 0 \rar \H_r(\xx;M) \otimes S[-r] \rar \cdots \rar \H_1(\xx; M) \otimes S[-1] \rar \H_{\m}^0(\H_0(\xx;M)) \otimes
 S \rar \H_{\m}^0(\gr_{(\xx)}(M)) \rar 0 \]
as asserted. 
\end{proof}

Several properties of the graded $\gr_{(\xx)}(R)$-module  $\H = \H^0_{\m}(\gr_{(\xx)}(M))$ are read off
 the exactness of the complex. More explicitly,
  $\H$  is a module over $(R/J)[\TT]$,
 $J$ taken as an $\m$-primary ideal annihilating the coefficient modules
of the complex. %Note that $\H / (\TT)\H = \H_{\m}^0(M/(\xx)M)$.
 Let us first explain the vanishing of $\H$ and the (maximal) Cohen-Macaulayness 
 of $\H$.

 \begin{Corollary} \label{vanishingofH} 
 %Let $(R, \m)$ be a Noetherian local ring and $M$ a finite $R$-module of dimension $d > 0$.  
 Let  $I=(\xx)$ be the ideal of $R$, where $\xx = x_1, \ldots, x_r$  is an amenable partial system of parameters of $M$ that is a $d$-sequence relative to $M$. Let $\TT= \TT_1, \ldots, \TT_r$ be variables, $S=R[\TT]$ the polynomial ring and $T=(R/J)[\TT]$, where $J$ taken as an $\m$-primary ideal annihilating the coefficient modules of the complex given in Theorem {\rm \ref{jdseqcx}}. Then we have the following.
 \begin{itemize}
 \item[{\rm (1)}] $\H_{I}(M)=(0)$ if and only if $\depth  M/IM >0$.
 \item[{\rm (2)}] $\H_{I}(M)$ is a Cohen-Macaulay $T$-module of dimension $r$ if and 
 and only if $\xx$ is an $M$-regular sequence and $\depth  M/IM = 0$.
 \item[{\rm (3)}] The $\TT$-depth of $\H_{I}(M)$ is equal to the $\xx$-depth of $M$.
\end{itemize} 
 \end{Corollary}

\begin{proof} (1) Suppose that $\H_{I}(M)=(0)$. Then by Theorem \ref{jdseqcx}
we have an exact complex
\[ 0 \rar \H_r(\xx;M) \otimes S[-r] \rar \cdots \rar \H_1(\xx; M) \otimes S[-1] \rar \H_{\m}^0(\H_0(\xx;M)) \otimes
 S   \rar 0\]
of graded $S$-modules. Reduction mod $(\TT)$ yields  an acyclic  complex of $R$-modules graded in different degrees, which therefore must be trivial. Conversely, suppose that $\depth  M/(\xx)M > 0$. Then $\H_{\m}^0(\H_0(\xx;M)) = \H_{\m}^0(M/(\xx)M) = (0)$, so that 
$\H_{I}(M)$ vanishes by Theorem \ref{jdseqcx}.

(2) The same argument will show that $\H_{I}(M) \cong \H_{\m}^0(M/(\xx)M)\otimes S$.

(3) By \cite[Proposition 3.8]{HSV3II}, the $\xx$-depth $M$ is equal to the $\TT$-depth of $\gr_{(\xx)}(M)$. This value is also obtainable from the $\TT$-depth of the image of $\H_1(\xx;M)\otimes S[-1]$ in $\H_0(\xx;M) \otimes S$.
\end{proof}

\begin{Corollary}\label{jdseqCor} 
Let $\xx= x_1, \ldots, x_r $~$(0 \leq r \le d)$ be an amenable partial system of parameters of $M$ that is a $d$-sequence relative to $M$. 
We set
\[ h_{i} =  \left\{ \begin{array}{ll} {\ds \lambda(\H_i(\xx;M)) } & \quad \mbox{\rm if}\;\; i >0,\\ &\\
  \lambda(\H_{\m}^0(\H_0(\xx;M)) & \quad \mbox{\rm if}\;\; i=0.\\
  \end{array} \right.\]
We then have the following. 
\begin{itemize}
\item[{\rm (1)}]  The   Hilbert series $[\![\H]\!]$ of the $\mathbf{j}$-transform $\H = \H_{\m}^0(\gr_{(\xx)}(M))$ is
\[ [\![\H]\!]= \frac{\sum_{i=0}^r (-1)^i h_i \ttt^i}{(1-\ttt)^{r}} \]
and the $\mathbf{j}$-polynomial of $M$ with respect to $\xx$ is
\[ \sum_{i=0}^r (-1)^i \rmj_i(\xx;M) {{n+r-i}\choose{r-i}}, \ \ \mbox{\rm where} \;\;
\rmj_{i}(\xx;M)= \sum_{k=i}^{r} (-1)^{k} h_{k} {{k}\choose{i}}.\]
In particular
\[ 
\rmj_0(\xx;M) = {\ds  \sum_{k=0}^r (-1)^{k} h_{k}} \ \  \mbox{and} \ \
\rmj_1(\xx;M) = {\ds \sum_{k=1}^r(-1)^{k} \cdot k \cdot  h_{k}.}
\]
\item[{\rm (2)}] $\xx$ is an $M$-regular sequence if and only if $\rmj_1(\xx;M) = \cdots = \rmj_r(\xx;M) =0$.
%\item[{\rm (3)}] $\H_{\m}^1(\gr_{(\xx)}(M)) \cong \H_{\m}^1(\H_0(\xx;M))\otimes S$. In particular, every superficial
%element $f$ of $(\xx)$ relative to $M$ induces an injective multiplication on $\H_{\m}^1(\gr_{(\xx)}(M))$.
%\item[{\rm (4)}] If $\dim \H =r$, then $j_1(\xx;M)\leq 0$. 
\end{itemize}
\end{Corollary}

\begin{proof}
(1) It is  read directly  from the acyclic complex given in Theorem \ref{jdseqcx}.
 
(2) The vanishing of the $\rmj_i(\xx;M)$'s for $i\geq 1$  is equivalent to the vanishing of the $h_i$'s for $i \geq 1$.
\end{proof}

\begin{Question} {\rm Let $\H=\H^{0}_{\m}( \gr_{(\xx)}(M) )$, where $\xx$ is an amenable partial system of parameters of $M$ that is a $d$-sequence relative to $M$. As given in the proof of Corollary \ref{j1sig}, we have $\rmj_1(\xx; M) = \e_1(\xx^*, \H) \leq 0$, where $\ff = f_1, \ldots, f_r$ denote the sequence of the initial forms $x_i^*$'s of $x_i$'s. Hence, if $\H$ is unmixed of dimension $r$, then the vanishing of $\rmj_1(\xx;M)$  means that $\H$ is Cohen-Macaulay  by \cite[Theorem 3.1]{chern5}. It will follow that $ \lambda(\H_i(\xx;M))=0$ for all $i>0$ and thus
$\xx$ is $M$-regular by Theorem \ref{jdseqcx} and Corollary \ref{vanishingofH}). How do we guarantee that $\H$ is unmixed? We need a condition on $M$ such that $\gr_{(\xx)}(M)$  satisfies Serre's condition ($S_1$).   Then $\H$, being a submodule,
inherits the condition. }
\end{Question}

\begin{Example}\label{j1example}{\rm 
Let $R$ be a Cohen-Macaulay local ring, $F$  a finite free $R$-module, and $M$ a submodule of $F$ such that
$C = F/M$ has finite length. Let $\xx=x_1, \ldots, x_r$, $0< r \le \dim R$,
 be a  partial system of parameters of $R$ that is an amenable $d$-sequence relative to  $M$ such that $(\xx)C=(0)$. Then the  commutative diagram
\[
\diagram
0 \rto  & (\xx)^{n+1}M \rto\dto & (\xx)^{n+1}F \rto \dto & C_{n+1}
\rto\dto  &0 \\
0 \rto          &  (\xx)^nM \rto                 & (\xx)^nF \rto  & C_{n} \rto  &0
\enddiagram
\]
with exact rows
yields the exact sequence
\[ 0 \rar C_{n+1} \lar (\xx)^nM/(\xx)^{n+1}M \lar (\xx)^nF/(\xx)^{n+1}F \lar C_{n} \rar 0,\] 
since the rightmost vertical mapping is trivial. Hence $[\H_{\m}^0(\gr_{(\xx)}(M))]_n = C_{n+1}$ for all $n \ge 0$.  Corollary~\ref{jdseqCor}  gives  a formula for the $\mathbf{j}$-polynomial in terms of the
Koszul homology modules $\H_i(\xx;M)$. Let us try to extract it from the exact sequence
$$0\rar M \rar F \rar C \rar 0.$$ We have the isomorphisms
\[
\H_{i}(\xx;M)  \cong  \H_{i+1}(\xx;C) = \wedge^{i+1} (R^{\oplus r})\otimes C\ \ \mbox{for} \ \ i\geq 1,\\
\]
so that
\[
\lambda(\H_i(\xx;M))  = \lambda(C){{r}\choose{i+1}}\ \ \mbox{for} \ \ 1 \leq  i \leq r-1 \quad \mbox{\rm and} \quad \lambda(\H_r(\xx;M)) =0.
\]
Also from 
 the exact sequence
  \[ 0 \rar \H_1(\xx;C) = C^{\oplus r} \lar \H_0(\xx;M) \lar \H_0(\xx;F) \rar \H_0(\xx;C) = C \rar 0, \]
  we get
 \[
 \H_{\m}^0(\H_0(\xx;M)) =C^{\oplus r}.
 \] By Corollary \ref{jdseqCor}  the Hilbert series $[\![\H]\!]$ of $\H_{\m}^0(\gr_{(\xx)}(M))$ is 
\[ [\![ \H]\!]= {\frac{\sum_{i=0}^{r}(-1)^{i}h_i \ttt^i}{(1-\ttt)^{r}}} 
= \lambda(C) {\frac{\sum_{i=0}^{r-1} (-1)^{i} {{r}\choose {i+1}}\ttt^i}{(1-\ttt)^{r}}}.\]
We furthermore have 
\[ \rmj_0(\xx;M) = \lambda(C) \sum_{i=0}^{r-1} (-1)^{i} {{r}\choose{i+1}} = \lambda(C)\]
and
\[ \rmj_{1}(\xx; M)= \lambda(C) \sum_{i=1}^{r-1}(-1)^{i}  { {r}\choose{i+1}} i = - \lambda(C).\]
}\end{Example}

\begin{Corollary} Let $R$ be a Cohen-Macaulay local ring of dimension $d\geq 3$ and $M$ a finite torsion-free $R$-module of projective dimension two. Let $\xx = x_1, \ldots, x_r $ $(0 < r \le d)$ be a partial system of parameters of $R$ and $M$ that is an amenable  $d$-sequence relative to $M$. We then have the  following.
\begin{itemize}
\item[{\rm (1)}]
 If $r\leq d-2$, then $\xx$ is an $M$-regular sequence.
\item[{\rm (2)}]
 If $r\geq d-1$, then $\rmj_1(\xx;M)\neq 0$.
\end{itemize}
\end{Corollary} 

\begin{proof} (1) This follows from Proposition \ref{amenablearereg}.

(2) If
$r=d$, then $\rmj_1(\xx;M) = \e_1(\xx;M)$ which cannot vanish, since $M$ is unmixed but not Cohen-Macaulay (\cite[Theorem 3.1]{chern5}). Suppose that $r=d-1$. Let $K$ be the module of syzygies of $M$:
\[ 0 \rar K \lar F \lar M \rar 0.\]
Since $\H_i(\xx;F)=(0)$, $i\geq 1$, we have $\H_i(\xx;K) = \H_{i+1}(\xx;M)$ for $i\geq 1.$
This means that $\xx$ is amenable for $K$. By Proposition \ref{amenablearereg} $\xx$ is a regular sequence on $K$. In particular, $\H_{i+1}(\xx;M)=0$ for all $i \geq 1$. 
Since $\depth M = d-2$, it follows that 
$\H_{1}(\xx;M)$ is the only non-vanishing Koszul homology module.
By Corollary~\ref{jdseqCor} (1) we obtain $\rmj_1(\xx;M) = -\lambda(\H_{1}(\xx;M)) \neq 0$.  
\end{proof}

Let us give another application.

\begin{Corollary} Let $R$ be  a Noetherian local ring and $M$  a Buchsbaum $R$-module of dimension $d > 0$. Let $0 < r \le d$ be a positive integer. Assume that $\dim \H_{(\xx)}(M) = s$ and $\rmj_1(\xx;M)=0$ for every
partial system $\xx = x_1, \ldots, x_s$ of parameters of $M$ of length $s$, $0 < s\leq r$. Then $M$ satisfies the condition
$(\mathrm{S}_r)$ of Serre.
\end{Corollary}

\begin{proof}
 We argue by induction on $r$. We may assume $r>1$ by Corollary \ref{jdseqCor} (1). By induction
on the case $r-1$, we have $\H_i(\xx;M) = 0$ for $i>1$, so  that $\rmj_i(\xx;M) = 0 $ for $i>1$. Therefore, if
$\rmj_1(\xx;M)=0$ as well, then by Corollary \ref{jdseqCor} (2) $\xx$ is an $M$-regular 
sequence.
  \end{proof}

\section{Local cohomology: Additional formulas for $\mathbf{j}$-coefficients}

Let $(R,\m)$ be a Noetherian local ring and let $M$ be a finite $R$-module of dimension $d > 0$. Throughout this section, let $\xx= x_1, \ldots, x_r$ ($0 \le r \le d$) be an amenable partial system of parameters of $M$ that is a $d$-sequence relative to $M$. We set $I = (\xx)$ and $\rmG = \gr_I(R)$.   Notice that when $r > 0$, the sequence $\xx'= x_2, \ldots, x_r$ is naturally an amenable partial system of parameters of $M/x_1M$ that is a $d$-sequence relative to $M/x_1M$.  Let $I_i = (x_1, \ldots, x_i)$ for each $0 \le i \le r$.

We begin with the following.

\begin{Lemma}\label{4.4} The following assertions hold true.
\begin{enumerate}
\item[$(1)$] Suppose that $r > 0$ and let $\overline{M} = M/x_1M$. Then there is a long exact sequence
\begin{eqnarray*}
0\to\rmH_\m^0(M)  \to  \rmH_\m^0(\overline{M}) &\to& \rmH_\m^1(M) \overset{x_1}{\to} \rmH_\m^1(M) \to \rmH_\m^1(\overline{M})  \\
&\to&  \rmH_\m^2(M) \overset{x_1}{\to} \rmH_\m^2(M) \to  \rmH_\m^2(\overline{M}) \to \cdots \\
&\to& \rmH_\m^i(M) \overset{x_1}{\to} \rmH_\m^{i}(M) \to  \rmH_\m^{i}(\overline{M}) \to \cdots
\end{eqnarray*}
of local cohomology modules.
\item[$(2)$]
Let $1 \le p \le r$ be an integer. Then $h^0(M/I_{p-1}M)= 0$ if and only if $x_1, \ldots, x_p$ is an $M$-regular sequence.
\end{enumerate}
\end{Lemma}

\begin{proof} (1) This follows from the fact that $ \rmH^0_\fkm(M) =  (0):_Mx_1 = (0):_MI$.

(2) It suffices to check the {\it only if part}. Since $(0):_Mx_1 = \rmH_\fkm^0(M)$, we may assume that $p > 1$ and that the assertion holds true for $p-1$. We consider the exact sequence
\[ 0 \to \rmH_\fkm^0(M/I_{p-2}M) \to \rmH_\fkm^0(M/I_{p-1}M) \to \rmH_\fkm^1(M/I_{p-2}M) \overset{x_{p-1}}{\to} \rmH_\fkm^1(M/I_{p-2}M) \to \cdots\] given by assertion (1). Then since $h^0(M/I_{p-1}M)= 0$, we get $h^0(M/I_{p-2}M)= 0$, so that by the hypothesis of induction the sequence $x_1, \ldots, x_{p-1}$ is $M$-regular, while $x_{p}$ is $M/I_{p-1}M$-regular, as $(0):_{M/I_{p-1}M} x_{p} = \rmH^0_\fkm(M/I_{p-1}M) = (0)$.
\end{proof}

For each $R$-module $C$, we set  $\rmG(C) = \gr_{I}(C)$ and let $\rmH (C) =\rmH_\fkm^0(\rmG (C))$ denote the {$\mathbf{j}$}-transform of $C$ relative to $I$.

\begin{Proposition}\label{Hvanishing}
The  following conditions are equivalent.
\begin{enumerate}
\item[{\rm (1)}] $\H(M)=(0)$.
\item[{\rm (2)}] $\rmH_\fkm^0(M/IM) = (0)$.
\item[{\rm (3)}] $\depth M \geq r+1$.
\end{enumerate}
When this is the case, $r < d$ and $\xx$ is an $M$-regular sequence.
\end{Proposition}

\begin{proof}(1) $\Leftrightarrow$ (2) See Theorem \ref{14}.

$(1) \Rightarrow (3)$ We have $r < d$, since $H \ne \rmG (M)$. On the other hand, $x_1, x_2 \ldots, x_r$ is an $M$-regular sequence, because 
\begin{eqnarray*}
[I_{i-1}M:_Mx_{i}]/I_{i-1}M &\cong& \left[I_{i-1}M :_M x_{i}+IM\right]/IM \\
&\subseteq& \rmH_\fkm^0(M/IM) = (0)
\end{eqnarray*}
for all $1 \le i \le r$ (remember that $\xx$ is a $d$-sequence relative to $M$ which is amenable for $M$). 
Therefore $x_1, x_2, \ldots, x_r$ is an $M$-regular sequence, whence   $\depth M > r$, as $\rmH_\fkm^0(M/IM) = (0)$.

$(3) \Rightarrow (2)$  This is clear.
\end{proof}

Let $W = \rmH_\m^0(M)$ and let  $\psi :\rmG (M) \to \rmG (M/W)$ be the canonical epimorphism of graded $G$-modules. We set $W^* = \Ker ~\psi$. Then since $W \cap IM = (0)$, we have $$W^*=\{w + IM \mid w \in W\} \subseteq M/IM = [\rmG (M)]_0,$$ so that $W^* = [W^*]_0 \cong W$ as an $R$-module.

\begin{Lemma}\label{ind1}
\begin{itemize}
\item[{\rm (1)}]  There is an exact sequence 
\[ 0 \to W^* \to \rmH (M) \overset{\varphi}{\to} \rmH (M/W) \to 0, \] 
of graded $\rmG$-modules, where $\varphi$ denotes the homomorphism induced from the canonical epimorphism $\psi: \rmG (M) \to \rmG (M/W)$.

\item[{\rm (2)}] Suppose  that $r > 0$ and let $\overline{M} = M/x_1M$. Then there is an exact sequence
$$0 \to W^*[-1] \to \rmH (M)[-1] \overset{f_1 } \to \rmH (M) \overset{\tau}{\to} \rmH(\overline{M}) \to 0$$
of graded $\rmG$-modules, where $\tau$ denotes the homomorphism induced from the canonical epimorphism $\sigma :\rmG (M) \to \rmG (\overline{M})$ and $f_{1}$ denotes the initial form of $x_{1}$.
\end{itemize}
\end{Lemma}

\begin{proof} 
(1) Apply the functors $\rmH_\m^i(-)$ to the  exact sequence 
$$0 \to W^* \to \rmG (M) \overset{\psi}{\to} \rmG (M/W) \to 0$$ of graded $\rmG$-modules  and we get the exact sequence as asserted, since $\rmH_\fkm^0(W^*) = W^*$ and $\rmH_\m^1(W^*) = (0)$.

(2) Remember that $\rmG (M)/f_1 \rmG (M) \cong \rmG (\overline{M})$, since $x_1M \cap I^{n+1}M = x_1I^nM$ for all $n \in \Bbb Z$ and that $(0):_{\rmG (M)}f_1  = W^*$, since $[(0):_M{x_1}] \cap I^nM  = (0)$ for all $n >0$. We then have the exact sequence $$0 \to W^*[-1] \overset{\iota}{\to} \rmG (M)[-1] \overset{f_1 } \to \rmG (M) \overset{\sigma}{\to} \mathrm{\rmG}(\overline{M}) \to 0$$of graded $\rmG$-modules, and applying the functors $\rmH_\m^i(-)$ to the sequence, we get the required exact sequence, because  by Theorem \ref{14} the homomorphism $\tau : \rmH (M) \to \rmH (\overline{M})$ is surjective. 
\end{proof}

For an $R$-module $C$ and $i \in \Bbb Z$ we set $h^i(C) =\lambda (\rmH_\fkm^i(C))$. Then by Lemma \ref{ind1} an induction on $r$ gives the following description of the Hilbert series $[\![\H]\!]$ of $\rmH=\rmH(M)$. We note a brief proof. 

\begin{Theorem}\label{22} 
%Let $(R, \m)$ be a Noetherian local ring, $M$ a finite $R$-module of dimension $d>0$. 
Let $\xx= x_1, \ldots, x_r$ {\rm(}$0 \leq r \le d${\rm)} be an amenable partial system of parameters of $M$ that is a $d$-sequence relative to $M$. We set $I = (\xx)$. Then the Hilbert series of the $\mathbf{j}$-transform $ \H=\H(M)$ is
\[ [\![\H]\!] =\frac{\ds h^0(M/IM) + \displaystyle \sum_{i=1}^{r}(-1)^i \left[\sum_{j=i}^rh^0(M/I_{r-j}M)\binom{j-1}{i-1}\right]\ttt^i}{\ds (1-\ttt)^{r}}.\]
\end{Theorem}

\begin{proof} We may assume that $r>0$ and the assertion holds true for $r-1$. Let $\overline{M}=M/x_{1}M$.
We write
\[ [\![\H]\!]  = \frac{ Q( \ttt) }{ (1-\ttt)^{r}}\]
with $Q( \ttt) \in \Bbb Z [\ttt]$. Then by the exact sequence
\[0 \to W^*[-1] \to \rmH (M)[-1] \to \rmH (M) \to \rmH(\overline{M}) \to 0\]
given in Lemma \ref{ind1} (2) and by the hypothesis of induction we get
\[ 0= h^{0}(M) \ttt - \frac{ Q(\ttt) \ttt}{(1-\ttt)^{r}} + \frac{ Q(\ttt) }{ (1-\ttt)^{r} } -   \frac{\ds h^0(\overline{M}/J \overline{M}) + \sum_{i=1}^{r-1}(-1)^i \left[\sum_{j=i}^{r-1} h^0( \overline{M}/J_{r-1-j} \overline{M})\binom{j-1}{i-1}\right]\ttt^i}{\ds (1-\ttt)^{r-1}},\]
where $J=(x_{2}, \ldots, x_{r})$ and $J_{r-1-j} = (x_{2}, \ldots, x_{r-j})$ for $1 \leq j \leq r-1$.
Since 
\[ M/IM \cong \overline{M}/J \overline{M}  \quad \mbox{\rm and} \quad  M/I_{r-j}M \cong   \overline{M}/J_{r-1-j} \overline{M} \;\; \mbox{\rm for} \;\; 1 \leq j \leq r-1, \] we obtain 
\[ \begin{array}{rcl}
Q(\ttt) &=& {\ds  h^{0}(M/IM)  + \sum_{i=1}^{r-1}(-1)^i \left[\sum_{j=i}^{r-1} h^0( M/I_{r-j} M)\binom{j-1}{i-1}\right]\ttt^i  - h^{0}(M)\ttt(1-\ttt)^{r-1}} \\
&& \\
&=& {\ds  h^{0}(M/IM)  + \sum_{i=1}^{r-1}(-1)^i \left[\sum_{j=i}^{r-1} h^0( M/I_{r-j} M)\binom{j-1}{i-1}\right]\ttt^i  + h^{0}(M) \sum_{i=1}^{r}(-1)^{i} \binom{r-1}{i-1}  \ttt^{i} }  \\ && \\
&=& {\ds h^{0}(M/IM)  + \sum_{i=1}^{r-1}(-1)^i \left[\sum_{j=i}^{r} h^0( M/I_{r-j} M)\binom{j-1}{i-1}\right]\ttt^i    + (-1)^{r} h^{0}(M) \ttt^{r}  }, \\ 
\end{array}\]
as claimed.
\end{proof}

\begin{Remark}\label{jformula}{\rm
Comparing Theorem \ref{22} with the formula  given by Corollary \ref{jdseqCor}, we have  $$\lambda(\rmH_i(\xx ;M))=\sum_{j=i}^r h^0(M/I_{r-j}M)\binom{j-1}{i-1}$$
 for all $1 \le i \le r$, which one can prove also by induction on $r$.
}\end{Remark}

We set 
\[ k_i( \xx;  M) = \left\{
\begin{array}{ll}
h^0\left(M/[(I_{r-i -1}M:_M{x_{r - i}})+ x_{r-i}M]\right) & \quad \mbox{if} \quad 0 \le i \le r-1,
\vspace{4mm}\\
h^0(M)& \quad \mathrm{if}  \quad i=r.
\end{array}
\right.\]  We then have the following.

\begin{Lemma}\label{4.6} $k_i (\xx; M) = h^0(M/I_{r-i}M) - h^0(M/I_{r-i - 1}M)$ for all $0 \le i \le r-1$.
Hence $k_{i}(\xx' ; \overline{M})= k_{i}(\xx; M)$ for all $0 \leq i \leq r-2$, where $\xx'=x_{2}, \ldots, x_{r}$ and $\overline{M}=M/x_{1}M$.
\end{Lemma}

\begin{proof} 
Set $L = M/I_{r - i - 1}M$. Then we get the short exact sequence 
$$0 \to \rmH^0_{\m} (L) \to L/ x_{r-i}L \to L/[\rmH^0_{\m} (L) + x_{r-i}L] \to 0.$$  Notice that
$\rmH^0_{\m} (L) = [I_{r - i-1}M:_M  x_{r-i} ]/I_{r - i-1}M$ 
and we have  $$ k_{i}(\xx ; M) = h^0\left(M/[I_{r-i -1}M:_M{x_{r - i}}+ x_{r-i}M]\right)= h^0(M/I_{r-i}M) - h^0(M/I_{r-i - 1}M),$$ because $h^0\left(M/(I_{r-i -1}M:_M{x_{r - i}}+ x_{r-i}M\right) = \lambda(L/[\rmH^0_{\m} (L) + x_{r-i}L])$. 
  \end{proof}

We now have the following formula of the $\mathbf{j}$-function of $M$.

\begin{Theorem}\label{19} 
%Let $(R,\m)$ be a Noetherian local ring and  $M$ a finite $R$-module of dimension $d > 0$. 
Let $\xx= x_1, \ldots, x_r$ {\rm(}$0 \le r \le d${\rm)} be an amenable partial system of parameters of $M$ that is a $d$-sequence relative to $M$. 
Then the $\mathbf{j}$-function of $M$ relative to $I=(\xx)$ is
\[ \psi^{M}_{I} (n) = \sum_{i=0}^{r}k_{i}(\xx; M)\binom{n+r-i}{r-i}\] for $n \ge 0$. Therefore $(-1)^{i} \rmj_{i}(\xx; M)= k_{i}(\xx;M)$ for all $0 \leq i \leq r$.
\end{Theorem}

\begin{proof} 
If $r=0$, then the $\mathbf{j}$-function of $M$ is $\psi^{M}_{(0)}(n) = \sum_{\ell=0}^n h^{0} (I^{\ell} M/I^{\ell+1}M) = h^0(M)$ for $n \geq 0$. Assume that $r > 0$ and the assertion holds true for $r-1$. Let $\overline{M}=M/x_{1}M$ and $\xx'=x_{2}, \ldots, x_{r}$. Then by Lemma \ref{ind1}, the hypothesis of induction, and Lemma \ref{4.6} we get
\[ \begin{array}{rcl}
\psi^{M}_{I}(n) &=& {\ds  \sum_{\ell =0}^n  h^{0} (I^{\ell} M/I^{\ell +1}M) } \\ && \\
                           &=& {\ds h^{0}(M/IM) +  \sum_{\ell=1}^n  h^{0} (I^{\ell} M/I^{\ell+1}M) } \\ && \\
                           &=& {\ds \psi^{\overline{M}}_{(\xx')}(0) + \sum_{\ell=1}^{n}  \left(  \psi^{\overline{M}}_{(\xx')}(\ell) - h^{0}(M) \right) } \\ && \\
                           &=& {\ds  \sum_{\ell=0}^{n} \left(  \psi^{\overline{M}}_{(\xx')}(\ell) \right) - n \cdot h^{0}(M)       } \\ && \\
                           &=& {\ds \sum_{\ell=0}^{n} \left( \sum_{i = 0}^{r-2}k_i(\xx'; \overline{M})\binom{\ell+r-1-i}{r-1-i} +   k_{r-1}(\xx'; \overline{M})    \right) - n \cdot h^{0}(M) } \\ && \\
                           &=& {\ds  \sum_{i = 0}^{r-2} k_i(\xx'; \overline{M})
\left( \sum_{\ell=0}^{n} \binom{\ell+r-i-1}{r-i-1} \right) + (n+1) h^{0}(\overline{M})   - n \cdot h^{0}(M) } \\ && \\
&=& {\ds  \sum_{i = 0}^{r-2} k_i(\xx'; \overline{M}) \binom{n + r - i}{r - i} + (n+1)( h^{0}(\overline{M}) - h^{0}(M) ) + h^{0}(M) } \\ && \\
&=& {\ds \sum_{i = 0}^{r-2} k_i(\xx; M) \binom{n + r - i}{r - i} + k_{r-1}(\xx; M)  \binom{n+1}{1} + k_{r}(\xx; M),} \\
 \end{array} \]
as wanted.
\end{proof}

\begin{Corollary}\label{20} Suppose that $r > 0$. Then 
\[ - \rmj_1(\xx ; M) = \left\{
\begin{array}{ll}
h^0(M) & \quad \mathrm{if} \quad r = 1,
\vspace{4mm}\\
k_1(\xx; M) = h^0\left(M/[(I_{r-2}M:_M{x_{r - 1}})+ x_{r-1}M]\right)& \quad \mathrm{if}  \quad r > 1,
\end{array}
\right.\]
whence $\rmj_1(\xx,M) \le 0$.

\end{Corollary}

\begin{Corollary}\label{21} Let  $r >0$. Then the  following assertions hold true.\begin{enumerate}
\item[{\rm (1)}] Suppose $r = 1$. Then $\rmj_1(\xx ; M) = 0$ if and only if $\depth M >0$.
\item[{\rm (2)}] Suppose $r >1$. Then $\rmj_1(\xx ; M)=0$ if and only if 
\[ \rmH_{\m}^0\left(  M/ [   (I_{r-2}M:_{M} x_{r-1})+x_{r-1}M ] \right) =(0).\]
\end{enumerate}
Therefore, if $\xx=x_1, \ldots, x_r$ is an $M$-regular sequence,  then $\rmj_1(\xx ; M) = 0$.  
\end{Corollary}

The following gives a partial answer to Conjecture \ref{j1is0}. 

\begin{Theorem}\label{21-1}  Let $\xx= x_1, \ldots, x_r$ be an amenable partial system of parameters of $M$ that is a $d$-sequence relative to $M$. Suppose that either {\rm (a)} $r \ge 2$ and $\depth M \ge r-1$ or that {\rm (b)} $r = 3$ and $M$ is unmixed. Then 
 $\rmj_1(\xx ; M) = 0$ if and only if $\xx=x_{1}, \ldots, x_{r}$ is an $M$-regular sequence.
\end{Theorem}

\begin{proof} {\bf Case (a).} Suppose that $r \ge 2$, $\depth M \ge r-1$, and $\rmj_1(\xx ; M) = 0$. If $r = 2$, then $\depth M \geq 1$, whence
$$h^{0}(M/x_{1}M) = k_{1}(\xx, M) = - \rmj_{1}(\xx, M) = 0$$
by Lemma \ref{4.6}, 
so that by Lemma \ref{4.4} (2)  $\xx=x_{1}, x_{2}$ is $M$-regular. Assume that $r \ge 3$ and the assertion holds true for $r-1$. Then $x_1$ is $M$-regular, because $(0) :_{M} x_{1}  = H^{0}_{\m}(M) =(0)$, while Lemma \ref{4.6} and Theorem \ref{19} show 
\[ \rmj_{1}( \xx' ; \overline{M} ) = - k_{1}(\xx' ; \overline{M}) = - k_{1}(\xx; M)= \rmj_{1}(\xx ; M) =0,\]
where $\xx' = x_2, \ldots, x_r$ and $\overline{M}=M/x_{1}M$. 
Hence the hypothesis of induction  shows $\xx'$ is $\overline{M}$-regular, so that $\xx$ is $M$-regular.

{\bf Case (b).} Suppose that $r = 3$ and $M$ is unmixed. By Case (a) it is enough to show that $\depth M \geq 2$. We set $\overline{M} = M/x_1M$ and $L= \overline{M}/ [(0):_{\overline{M}}x_2]$. 
 Then $x_2$ is $L$-regular and $\depth L \geq 1$. Let $\xx'=x_{2}, x_{3}$ and notice that  
\[ \begin{array}{rcl}
- \rmj_{1}( \xx' ; L ) &=& h^{0}(L/x_{2}L) - h^{0}(L) = h^{0}(L/x_{2}L) \\ && \\
&=& h^{0}( \overline{M}/( (0:_{\overline{M}} x_{2})+ x_{2} \overline{M} ) ) = - \rmj_{1}(\xx', \overline{M}) = - \rmj_{1}( \xx, M) =0.
\end{array}\]
Then by Case (a) $ \xx'= x_{2}, x_{3}$ is $L$-regular, whence $\H^{1}_{\m}(L)=(0)$, so that $\H^{1}_\fkm(\overline{M})=(0)$ (remember that $L = \overline{M}/\rmH_\fkm^0(\overline{M})$). Since $\rmH_\fkm^1(M)$ is finitely generated as $M$ is unmixed, we get $\rmH_\fkm^1(M)=(0)$ by Nakayama's lemma, thanks to the exact sequence
\[\rmH_\m^1(M) \overset{x_1}{\to} \rmH_\m^1(M) \to \rmH_\m^1(\overline{M}) =0\]
 given in  Lemma \ref{4.4} (1). Hence $\depth M \ge 2$. 
\end{proof}

The condition (a) that $\depth M \ge r-1$ in Theorem \ref{21-1} is not too much. Let us explore one example.

\begin{Example}{\rm
Let $S$ be a Buchsbaum local ring of dimension $3$ and $a,b,c$ a system of parameters of $S$. We assume $h^0(S) \ne 0$ but $h^1(S) = 0$. Let $R = S[[X, Y]]$ be the formal power series ring and set $x_1 = X, x_2 = Y, x_3 = a, x_4=b, x_5 = c$. Then $x_1, x_2, x_3, x_4, x_5$ is a system of parameters of $R$, forming a $d$-sequence. Let $\xx = x_1, x_2, x_3, x_4$. Since $$R/[(x_1,x_2):x_3 +x_3R] \cong S/[(0):a + aS],$$ we get 
\begin{eqnarray*}
h^0(R/[(x_1,x_2):x_3 +x_3R]) &=& h^0(S/[(0):a + aS]) \\
&=& h^0(S/(0):a) + h^1(S/(0):a)\\
&=& h^1(S) \\
&=& 0,
\end{eqnarray*}
so that  $j_1(\xx; R) = 0$, while  $\depth R = 2$. Notice that $R$ is a mixed local ring.

Let $I = (\xx)$ and set $\rm\rmH = \rmH_{I}(R)$. Then $$\rmH \cong \left[\overline{G}/(f_3, f_4)\overline{G}\right]^{\oplus h^0(S)} \oplus \overline{G}^{\oplus h^2(S)}$$ as a graded $\overline{G}$-module, where $\overline{G} = G/\fkm G$.   Therefore $\rmH \ne (0)$ and \[ \dim \rmH = \left\{
\begin{array}{ll}
4 & \quad \mbox{if} \quad h^2(S) > 0, 
\vspace{4mm}\\
2 & \quad \mbox{if}  \quad h^2(S) =0.
\end{array}
\right.\] 
}\end{Example}

\begin{proof} We have $\fkm H = (0)$, since $\fkm \rmH^0_\fkm(R/I) = (0)$ (remember that $R/I \cong S/(a,b)$ is a Buchsbaum local ring).
%Let $\ell = h^0(S)$ and $m = h^2(S)$. 
 Let $J = (0):_S\fkn$, where $\fkn$ is the maximal ideal of $S$. Then $J + I = JR + I \subseteq I :\fkm$. We also have $J \cap I \subseteq [(X, Y ) :_R a] \cap (X,Y, a, b) = (X,Y)$, so that $J \cap I = (0)$, because $J \cap (X,Y) = (0)$. Therefore $[JR + I]/I \cong J$ as an $S$-module, whence $\ell_R([JR + I]/I) = \ell_S(J) = h^0(S)$. Consequently, because $(f_3,f_4) \left[(JR+I)/I\right] = (0)$ in $\rmH$, we get a homomorphism
$$\varphi : \left[\overline{G}/(f_3, f_4)\overline{G}\right]^{\oplus h^0(S)} \oplus \overline{G}^{\oplus h^2(S)} \to \rmH$$
of graded $\overline{G}$-modules, which sends $(\mathbf{e}_i~\mathrm{mod}~ (f_3, f_4)\overline{G}, 0)$ to $v_j$ and $(0, \mathbf{e}_k)$ to $w_k$, where $\{v_j, w_k\}$ is a $k$-basis of $[I:_R\fkm]/I$ which extends a $k$-basis $\{v_j\}$ of $[JR + I]/I$. Then $\varphi$ is surjective, since $\rmH = \overline{G}{\cdot}\rmH_0$ and $\rmH_0 = [I:_R\fkm]/I$. Therefore $\varphi$ must be an isomorphism, as the Hilbert series of both sides are the same, which follows from Theorem \ref{19} (remember that $h^0(R) = h^0(R/I_1)=0, h^0(R/I_2) = h^0(R/I_3)= h^0(S)$, and $h^0(R/I)=h^0(S) + h^2(S)$).
\end{proof}

\subsubsection*{Finite local cohomology} 
A finite $R$-module $M$ is said to have {\it finite local cohomology} (FLC), if $\lambda(\H_{\m}^i(M)) < \infty$ for all $i \ne  \dim M$. This condition is equivalent to saying that $M$ possesses a system of parameters which is a $d^+$-sequence relative to $M$ (see \cite{STC, SV}). We call such a system of parameters of $M$ is {\it standard} (\cite{T}). Remember that every partial system of parameters of $M$ is amenable, once $M$ has FLC (\cite{STC}). 

\begin{Proposition}\label{4.6FLC}  Suppose that $M$ has {\rm FLC} and that $\xx = x_1, \ldots, x_r$ is a part of a standard system of parameters of $M$. Then for each $0 \le i \le r-1$ we have
\[ (-1)^{i} \rmj_{i}(\xx; M)= k_i(\xx; M) = \sum_{j=1}^{r-i}\binom{r-i -1}{j-1}h^j(M).\]
\end{Proposition}

\begin{proof} By \cite[Lemma 4.14]{SV} $h^{0}(M/I_{k}M) = \sum_{j=0}^{k} {{k}\choose{j}} h^{j}(M)$ for $0 \leq k < d$.
Hence by Lemma \ref{4.6} and Theorem \ref{19} we get 
\[ \begin{array}{rcl}
(-1)^{i} \rmj_{i}(\xx; M)= k_{i}(\xx ; M) &=& h^{0}(M/I_{r-i}M) - h^{0}(M/I_{r-i-1}M) \\ && \\
                     &=& {\ds \sum_{j=0}^{r-i} {{r-i}\choose{j}} h^{j}(M)  - \sum_{j=0}^{r-i-1} {{r-i-1}\choose{j}} h^{j}(M)    } \\ && \\
                     &=& {\ds \sum_{j=1}^{r-i-1} \left( {{r-i}\choose{j}} - {{r-i-1}\choose{j}} \right) h^{j}(M) + h^{r-i}(M) } \\ && \\
                     &=& {\ds \sum_{j=1}^{r-i}\binom{r-i -1}{j-1}h^j(M)} \\
\end{array}\] \end{proof}

\begin{Theorem}\label{FLCj1}
Suppose that $M$ has {\rm FLC} and $\depth M > 0$. Let $\xx= x_1, \ldots, x_r$ be a partial system of parameters of $M$ that is a $d$-sequence relative to $M$. Then 
 $\rmj_1(\xx ; M) = 0$ if and only if $\xx=x_{1}, \ldots, x_{r}$ is an $M$-regular sequence.
\end{Theorem}

\begin{proof} We have only to show that $\xx$ is an $M$-regular sequence, if $j_1(\xx,M) = 0$ (Corollary \ref{21} (3)). Notice that $M$ is unmixed, since $\depth M > 0$ (\cite{STC}). Therefore by Theorem \ref{21-1} we may assume that $r \ge 4$ and that our assertion holds true for $r-1$. We set $\overline{M} = M/x_1M$, $L = \overline{M}/[(0):_{\overline{M}}x_2]$, and $\xx'=x_{2}, \ldots, x_{r}$.
 Then $L$ has FLC, $\depth L >0$, and 
 \[ \rmj_{1}( \xx' ; L) = \rmj_{1}(\xx' ; \overline{M}) = \rmj_{1}(\xx; M)=0.\] Therefore by the hypothesis of induction on $r$ the sequence $\xx'$ is $L$-regular, whence  $\depth L \geq r-1$. Thus $\rmH_\m^{i}(\overline{M}) =(0)$ for all $1 \leq i \leq r-2$, so that by the exact sequence
 \[ \begin{array}{rcl}
0\to\rmH_\m^0(M)  \to  \rmH_\m^0(\overline{M}) \to \rmH_\m^1(M) \overset{x_1}{\to} \rmH_\m^1(M) \to 0  &\to&  \rmH_\m^2(M) \overset{x_1}{\to} \rmH_\m^2(M) \to  0 \to \cdots \\ && \\
\to 0 &\to& \rmH_\m^{r-2}(M) \overset{x_1}{\to} \rmH_\m^{r-2}(M) \to  0  \to \cdots \\
\end{array}\]
given in Lemma \ref{4.4} (1) we get 
$\rmH_\m^{i}(M) =(0)$ for all $1 \leq i \leq r-2$ (use Nakayama's lemma). Thus $\depth M \geq r-1$.
\end{proof}

\subsubsection*{Partial Euler characteristics}  
We are now in a position to compare $\rmj_1(\xx; M)$ to  partial Euler characteristics.
Suppose that $r > 0$ and we set $$\chi_1(x_1, \ldots, x_k;M) = \displaystyle\sum_{i=1}^{k}(-1)^{i+1}\lambda(\rmH_i(x_1, \ldots, x_k;M))$$ for each $1 \le k \le r$.  We then have the following.

\begin{Lemma}\label{23a} Let $\xx= x_1, \ldots, x_r$ {\rm(}$0 < r \le d${\rm)} be an amenable partial system of parameters of $M$ that is a $d$-sequence relative to $M$. 
Let $1 \le k \le r$ be an integer. Then 
\[\chi_1(x_1, \ldots, x_k;M)=h^0(M/I_{k-1}M).\] Hence $\chi_1(x_1, \ldots, x_k;M)=0$ if and only if $x_1, \ldots, x_k$ is an $M$-regular sequence.
\end{Lemma}

\begin{proof}
By Remark \ref{jformula}
\begin{eqnarray*}
\chi_1(x_1, \ldots, x_k:M) &=&\displaystyle\sum_{1 \le i \le j \le k}(-1)^{i+1}h^0(M/I_{k-j}M)\binom{j-1}{i-1}\\
&=&\displaystyle\sum_{j=1}^kh^0(M/I_{k-j}M)\left[\displaystyle\sum_{i=1}^j(-1)^{i+1}\binom{j-1}{i-1}\right]\\
&=& \displaystyle\sum_{j=1}^kh^0(M/I_{k-j}M)\left[ \displaystyle\sum_{i=0}^{j-1}(-1)^i\binom{j-1}{i}\right]\\
&=&h^0(M/I_{k-1}M)
\end{eqnarray*} 
The second assertion follows from Lemma \ref{4.4} (2).
\end{proof}

The following results extend \cite[Theorems 3.7 and 4.2]{chern7} to partial amenable systems of parameters.

\begin{Corollary}\label{Eulerj1}
Let $\xx= x_1, \ldots, x_r$ {\rm(}$0 < r \le d${\rm)} be an amenable partial system of parameters of $M$ that is a $d$-sequence relative to $M$. 
Then for all $0 \leq i \leq r$
\[ (-1)^{i} \rmj_{i} ( \xx ; M) = \chi_{1}( x_{1}, \ldots, x_{r+1-i}; M) - \chi_{1}( x_{1}, \ldots, x_{r-i}; M).\]
\end{Corollary}

\begin{proof}
For $0 \leq i \leq r-1$ 
\[ \begin{array}{rcll}
 (-1)^{i} \rmj_{i} ( \xx ; M)  =  k_{i}( \xx; M) &=& {\ds  h^0(M/I_{r-i}M) - h^0(M/I_{r-i - 1}M) }  \quad & \mbox{\rm (by (\ref{4.6})) }\\ &&& \\
 & = &{\ds \chi_{1}( x_{1}, \ldots, x_{r+1-i}; M) - \chi_{1}( x_{1}, \ldots, x_{r-i}; M) } \quad & \mbox{\rm (by (\ref{23a})), } \\
 \end{array}\]
while for $i=r$ 
 \[ (-1)^{r} \rmj_{r} ( \xx; M) = k_{r}(\xx; M) = h^{0}(M) = \chi_{1}(x_{1} ; M),\]
  since $\rmH^{0}_{\m}(M)=  (0):_{M} x_{1}$. 
\end{proof}

We close this section with the following.

\begin{Proposition}\label{23}
Let $\xx= x_1, \ldots, x_r$ {\rm(}$ 2 \le r \le d${\rm)} be an amenable partial system of parameters of $M$ that is a $d$-sequence relative to $M$. 
Then
\[ \chi_1(\xx;M) \ge - \rmj_1(\xx; M), \]
where the equality holds true if and only if $x_1, \ldots, x_{r-1}$ is an $M$-regular sequence.
\end{Proposition}

\begin{proof}
Notice that 
\[ \chi_1(x_1,  \ldots, x_r;M) = \chi_1(x_1, \ldots, x_{r-1};M) - \rmj_1(\xx; M) \ge - \rmj_1(\xx; M)\]
(Corollary \ref{Eulerj1}) and we get $\chi_1(x_1, \ldots, x_{r};M)= - \rmj_1(\xx; M)$ if and only if $\chi_1(x_1, \ldots, x_{r-1};M) =0$. Hence  the assertion follows from Lemma \ref{23a}.
\end{proof}

\section{Boundedness of $\mathbf{j}$-coefficients}

Let $(R, \m)$ be a Noetherian local ring and $M$ a finite $R$-module of dimension $d = \dim M \geq 2$. Let $ 0 < r <d$ be an integer. 
In this section firstly we discuss the problem of when the set $\Lambda(M)$ of non-negative integers
$$k_i(\xx;M) = (-1)^i \rmj_i(\xx;M)$$  is finite, where $\xx= x_1, \ldots, x_r$ ($0 \le i \le r-1$) is an amenable partial system of parameters of $M$ which is a $d$-sequence relative to $M$. The goal is to prove the following.

\begin{Theorem}\label{Finjplus}  
Assume that there exists a system of parameters of $M$ which is a strong $d$-sequence relative to $M$. 
 Then the following conditions are equivalent.
\begin{enumerate}
\item[{\rm (1)}] The set $\Lambda(M)$ is finite.
\item[{\rm (2)}] $\rmH^i_\m(M)$ is a finite $R$-module for every $1 \le i \le r$.
\end{enumerate}
When this is the case, one has $\m^{\ell}\rmH_\m^i(M) = (0)$ for all $1 \le i \le r$, where $\ell = \max \Lambda(M)$.
\end{Theorem}

For the implication (2) $\Rightarrow$ (1) we do not need the assumption of  the existence of a system of parameters of $M$ which is a strong $d$-sequence relative to $M$.

\begin{proof}
(2) $\Rightarrow$ (1)
Let $\xx = x_1, \ldots, x_r $ be an amenable partial system of parameters of $M$ that is a $d$-sequence relative to $M$. 
 We will show that for all $0 \leq i \leq r-1$ 
$$k_i(\xx; M) \le \displaystyle\sum_{j=1}^{r-i}h^j(M)\binom{r-i -1}{j-1}.$$ If $r= 1$, then by Lemma \ref{4.6} and the exact sequence
$$(\sharp_0)\ \ \ 0 \to \rmH^0_\m(M) \to \rmH^0_\m(M/x_1M) \to \rmH^1_\m(M) \overset{x_1}{\to} \rmH^1_\m(M)$$
stated in Lemma \ref{4.4} (1) we have
\[ k_{0}(\xx; M) = h^0(M/x_1M) - h^0(M) = \lambda((0):_{\rmH^1_\m(M)}{x_1})\le h^1(M).\]
 Suppose that $r > 1$ and that our assertion holds true for $r-1$. 
We consider $\overline{M} = M/x_1M$. Then thanks to the exact  sequence $$(\sharp_i)\ \ \ \ \rmH_\fkm^i(M) \overset{x_1}{\to} \rmH^i_\m(M) \to \rmH^i_\m(\overline{M}) \to \rmH^{i+1}_\m(M) \overset{x_1}{\to} \rmH^{i+1}_\m(M)$$ given in Lemma \ref{4.4} (1), $\rmH^i_\m(\overline{M})$ is finitely generated and $h^i(\overline{M}) \le h^i(M)+ h^{i+1}(M)$ for all $1 \le i \le r-1$.
Let $\xx'=x_2, \ldots, x_r$. Then for all $0 \le i \le r-2$, by Lemma \ref{4.6} and the hypothesis of induction we get
\[
\begin{array}{rcl}
{\ds k_i(\xx; M) = k_i( \xx';  \overline{M}) \le \sum_{j=1}^{r-i-1}h^j(\overline{M})\binom{r-i -2}{j-1} }
&\le& {\ds \sum_{j=1}^{r-i-1}\left[h^j(M) + h^{j+1}(M)\right]\binom{r-i -2}{j-1} }\\
&=& {\ds \sum_{j=1}^{r-i}h^j(M)\binom{r-i -1}{j-1}, } \\
\end{array} 
\]
while we have by Lemma \ref{4.6}
$$k_{r-1}(\xx; M) = h^0(M/x_1M) - h^0(M) \le h^1(M),$$ thanks to  exact sequence
$(\sharp_0)$ above.
\bigskip

 (1) $\Rightarrow$ (2) We choose a system $a_1, \ldots, a_d$ of parameters of $M$ which is a strong $d$-sequence relative to $M$.
Let $\Lambda_0$ denote the set of $k_i(a_1^{n_1},a_2^{n_2}, \ldots, a_r^{n_r};M)'s$, where $0 \le i \le r-1$ and $n_i's$ are positive integers. Then $\Lambda_0 \subseteq \Lambda$ and hence $\Lambda_0$ is finite. We will show by induction on $r$ that $\m^{\ell}\rmH^i_\m(M) = (0)$ for all $1 \le i \le r$, where $\ell = \max \Lambda_0$.

Let $r=1$. Consider exact sequence $(\sharp_0)$ above where $x_{1}=a_{1}^{n_1}$ and $n_{1} >0$.
Then we have
$$\lambda \left( \left[(0):_{\rmH^1_\m(M)}a_1^{n_1}\right] \right)= h^0(M/a_1^{n_1}M) - h^0(M) = k_0(a_1^{n_1};M) \in \Lambda_0,$$ so that $\lambda \left( \left[(0):_{\rmH^1_\m(M)}a_1^{n_1}\right] \right)\le \ell$, whence  
$$\m^\ell\left[(0):_{\rmH^1_\m(M)}a_1^{n_1}\right] = (0).$$
 Therefore $\m^\ell\rmH^1_\m(M) = (0),$ since $n_1 >0 $ is arbitrary. Assume now that $r>1$ and that our assertion holds true for $r-1$. Let $\overline{M} = M/a_{1}^{n_1}M$.  Then the set of $$k_i(a_2^{n_2}, \ldots, a_r^{n_r}; \overline{M})=k_i(a_1^{n_1}, \ldots, a_r^{n_r}; M)$$
where $0 \le i \le r-2$ and $n_i's$ are positive integers is a subset of  $\Lambda_0$, whence  the hypothesis of induction shows
$\m^{\ell}\rmH^i_\m(\overline{M}) = (0)$
for all  $1 \le i \le r-1$. Since $\left[(0)  :_{\rmH^{i+1}_\m(M)}a_1^{n_1}\right]$ is a homomorphic image of $\rmH^i_\m(\overline{M})$ as shown in exact sequence $(\sharp_{i})$, we have
$$\m^\ell\left[(0)  :_{\rmH^{i+1}_\m(M)}a_1^{n_1}\right] = (0)$$
for all $1 \leq i \leq r-1$, while by exact sequence $(\sharp_0)$ we have
\[ \lambda \left( \left[(0)  :_{\rmH^{1}_\m(M)}a_1^{n_1}\right] \right)  = h^{0}(M/a_{1}^{n_1}M)- h^{0}(M) = 
k_{r-1}( a_1^{n_1},a_2^{n_2}, \ldots, a_r^{n_r};M)   \in \Lambda_{0}.\]
Therefore
\[ \m^\ell\left[(0)  :_{\rmH^{i}_\m(M)}a_1^{n_1}\right] = (0)  \quad \mbox{\rm for all} \;\; 1 \leq i \leq r. \]
Thus 
$\m^\ell\rmH^i_\m(M) = (0)$ for all $1 \le i \le r$, since $n_{1} >0$ is arbitrary.
\end{proof}

Let us describe in Proposition \ref{3} below a broad class of modules for which the existence of strong $d$-sequences is guaranteed. It is based
on a result of T. Kawasaki \cite[Theorem 4.2 (1)]{K}, which deals with a case of rings. We extend this result to a case of modules and include a brief proof for the reader's convenience. We start with the following observation.

\begin{Lemma}\label{2}
Let $0 \to X \to Y \to Z \to 0$ be an exact sequence of $R$-modules and let $f_1, \ldots, f_n$ be a sequence of elements in $R$. If $f_1, \ldots, f_n$ is a $Z$-regular sequence which is a $d$-sequence relative to $Z$, then $f_1, \ldots, f_n$ is a $d$-sequence relative to $X$.
\end{Lemma}

\begin{Proposition}\label{3}
Suppose that $R$ is a homomorphic image of a Gorenstein local ring and let $M$ be a finite $R$-module of dimension $d > 0$. Then there is a system of parameters of $M$ which is a strong $d$-sequence relative to $M$.
\end{Proposition}

\begin{proof}
We may assume that $R$ is a Gorenstein local ring with $\dim R = \dim M=d$ and consider the idealization $A = R \ltimes M$ of $M$ over $R$.  Then, 
since $A$ is a homomorphic image of a Gorenstein ring,
 by \cite[Theorem 4.2 (1)]{K} $A$ contains a system $f_1, \ldots, f_d$ of parameters which is a strong $d$-sequence relative to $A$. We write $f_i = (a_i, x_i)$ with $a_i \in \m$ and $x_i \in M$. Let $p : A \to R, (a,x) \mapsto a$ and let ${}_pM$ denote $M$ which is regarded as an $A$-module via the ring homomorphism $p : A \to R, (r,m) \mapsto r$. Then ${}_pM =(0) \times M  \subseteq A$ and we get the exact sequence
$$0 \to {}_pM \to A \overset{p}{\to} R \to 0$$
of $A$-modules. Since $R$ is a Gorenstein local ring possessing  $a_1, \ldots, a_d$ as a system of parameters,  by Lemma \ref{2} $f_1, \ldots, f_d$ is a strong $d$-sequence relative to ${}_pM$. Hence $a_1, \ldots,a_d$ is a strong $d$-sequence relative to $M$.
\end{proof}

We note a consequence of Theorem~\ref{Finjplus}.

\begin{Corollary}  Let $\xx= x_1, \ldots, x_r$ be an amenable partial system of parameters of $M$ which is a $d$-sequence relative to $M$.  Suppose that $\lambda(\rmH_\m^i(M)) < \infty$ for all $1 \le i \le r$. Let $0 \le  p < r$ be an integer. Then the following conditions are equivalent.
\begin{enumerate}
\item[{\rm (1)}] $\rmj_p(\xx ; M) = 0$.
\item[{\rm (2)}] $\rmH^i_\m(M) = (0)$ for all $1 \le  i \le r - p.$
\end{enumerate}
Therefore  $\rmj_1( \xx ; M) = 0$ if and only if $\xx$ is $M$-regular, provided $\depth M > 0$.
\end{Corollary}

\begin{proof} Suppose that $r=1$ and consider the exact sequence 
$$ 0 \to \rmH_\fkm^{0}(M)  \to \rmH^0_\m(M/x_{1}M) \to \rmH^{1}_\m(M) \overset{x_1}{\to} \rmH^{1}_\m(M).$$
Then $\rmj_{0}(\xx; M)= k_{0}(\xx; M) = h^0(M/x_{1}M) - h^0(M) = \lambda((0):_{\rmH^1_\m(M)}x_1)$ by Lemma \ref{4.6}, whence $\rmj_{0}(\xx; M)=0$ if and only if $\rmH^1_\m(M) = (0)$.

Suppose that $r >1$ and that our assertion holds true for $r-1$. Let $\overline{M}=M/x_{1}M$ and $\xx'=x_{2}, \ldots, x_{r}$. Then thanks to the exact sequence 
$$(\sharp_i)\ \ \ \ \rmH_\fkm^i(M) \overset{x_1}{\to} \rmH^i_\m(M) \to \rmH^i_\m(\overline{M}) \to \rmH^{i+1}_\m(M) \overset{x_1}{\to} \rmH^{i+1}_\m(M)$$
given by Lemma \ref{4.4} (1), $\rmH_\m^i(\overline{M})$ is finitely generated for all $1 \le i \le r-1$. Therefore for all $0 \leq p \leq r-2$ we get by Lemma \ref{4.6}, Theorem \ref{19}, and the hypothesis of induction that 
\[\rmj_{p}(\xx'; \overline{M})= \rmj_{p}(\xx; M) = 0   \]
if and only if
\[ \rmH_\m^i(\overline{M}) = (0) \;\; \mbox{\rm for all} \;\; 1 \leq i \leq r-1-p.\]
By exact sequence $(\sharp_{i})$ the latter condition is equivalent to saying that $\rmH_\m^i(M) = (0)$ for all $1 \le i \le r-p$, because each $\rmH_\m^i(M)~(1 \le i \le r)$ has finite length, while for $p=r-1$, by Lemma \ref{4.6} we get
\[ (-1)^{r-1} \rmj_{r-1}(\xx ; M)=k_{r-1}(\xx; M) = h^0(M/x_{1}M) - h^0(M) = \lambda((0):_{\rmH^1_\m(M)}x_1).\] Thus equivalence of assertions (1) and (2)  follows.

See Theorem \ref{21-1} for the last assertion. 
\end{proof}

We now consider the problem of when the set $\Gamma(M)$ of 
\[ \chi_1(\xx ;M) = \sum_{i=1}^r(-1)^{i-1}\lambda(\rmH_i(\xx;M)) = h^0(M/(x_1,  \ldots, x_{r-1})M)\]
is finite, where $\xx=x_1, \ldots, x_r$ is an amenable partial system of parameters of $M$ which is a $d$-sequence relative to $M$.

\begin{Theorem}\label{5.5} Assume that there exists a system of parameters of $M$ which is a strong $d$-sequence relative to $M$. Then the following conditions are equivalent.
\begin{enumerate}
\item[{\rm (1)}] The set $\Gamma(M)$ is finite.
\item[{\rm (2)}] $\rmH^i_\m(M)$ is finitely generated for every $0 \le  i \le r-1$.
\end{enumerate}
\end{Theorem}

For the implication (2) $\Rightarrow$ (1) we do not need the assumption of  the existence of a system of parameters of $M$ which is a strong $d$-sequence relative to $M$.

\begin{proof}
(2) $\Rightarrow$ (1) 
Let $\xx=x_1, \ldots, x_r$ be an amenable partial system of parameters of $M$ which is a $d$-sequence relative to $M$. 
% We consider $$\chi_1(\xx;M) = h^0(M/(x_1, \ldots, x_{r-1})M) \ge 0.$$
 We will show by induction on $r$ that 
$$\chi_1(\xx;M) \le \displaystyle\sum_{i=0}^{r-1}\binom{r-1}{i}h^i(M).$$ If $r = 1$, we have nothing to prove. Suppose that $r > 1$ and that our assertion holds true for $r-1$. Let $\overline{M} = M/x_1M$. Then by Lemma \ref{4.4} (1) $\rmH_\m^i(\overline{M})~(0 \le i \le r-2)$ is finitely generated with $h^i(\overline{M}) \le h^i(M) + h^{i+1}(M)$. Let $\xx'=x_2, \ldots, x_r$. Then
\[\begin{array}{rcl}
{\ds \chi_1(\xx;M) =\chi_1(\xx' ;\overline{M}) 
\le \sum_{i=0}^{r-2}\binom{r-2}{i} h^i(\overline{M})  }
&\le &  {\ds \sum_{i=0}^{r-2}\binom{r-2}{i} \left[h^i(M)+h^{i+1}(M)\right]}  \\
&=& {\ds  \sum_{i=0}^{r-1}\binom{r-1}{i} h^i(M).} \\
\end{array}
\]

 (1) $\Rightarrow$ (2) Let $a_1, \ldots, a_d$ be a system of parameters of $M$ which is a strong $d$-sequence relative to $M$ and let $\Gamma_0(M)$ be the set of all $\chi_1(a_1^{n_1},\ldots, a_r^{n_r};M)$ with integers $n_i >0$. Then $\Gamma_0(M) \subseteq \Gamma (M)$. We set $\ell = \max \Gamma_0(M)$ and show by induction on $r$ that 
$$\lambda(\rmH_\m^i(M)) \le \ell$$ for all $0 \le  i \le r-1$. If $r = 1$, then $h^0(M)= \chi_1(a_1^{n_1};M)   \in \Gamma_0(M)$ and the assertion is clear. Suppose that $r > 1$ and that the assertion holds true for $r-1$. Choose an arbitrary integer $n_1 > 0$ and let $\overline{M} = M/a_{1}^{n_{1}}M$ and consider the exact sequence
$$(\sharp_i)\ \ \ \cdots \to \rmH^i_\m(M) \overset{a_{1}^{n_{1}}}{\to} \rmH^i_\m(M) \to \rmH^i_\m(\overline{M}) \to \rmH^{i+1}_\m(M) \overset{x_1}{\to} \rmH^{i+1}_\m(M)\to \cdots.$$ 
Let $\Gamma_0(\overline{M})$ be the set of $\chi_1(a_2^{n_2}, \ldots, a_r^{n_r};\overline{M})$ with $n_i$ positive integers. Then since $$\chi_1(a_2^{n_2}, \ldots, a_r^{n_r};\overline{M}) = \chi_1(a_1^{n_1}, \ldots, a_r^{n_r};M),$$ $\Gamma_0(\overline{M}) \subseteq \Gamma_0(M)$, whence  the hypothesis of induction on $r$ guarantees  
$$\lambda(\rmH_\m^i(\overline{M})) \le \ell$$
for all $0 \le   i \le r-2$. From the exact sequence $(\sharp_i)$, since ${\ds \left[ (0):_{\rmH^{i+1}_\m(M)}  a_1^{n_{1}}\right]}$ is a homomorphic image of $\rmH^i_\m(\overline{M})$, we get $$\lambda\left( \left[  ( 0) :_{\rmH^{i+1}_\m(M)} a_1^{n_{1} } \right] \right) \le \ell$$ for all $0 \le i \le r-2$.  Since $n_{1}$ is arbitrary, this means that
\[ \lambda \left( \rmH^{i}_{\m} (M) \right) \leq \ell \;\; \mbox{\rm for all} \;\; 1 \leq i \leq r-1.\] Because $\rmH^0_\m(M) \subseteq \rmH^0_\m(\overline{M})$ by Lemma \ref{4.4} (1), we get $\lambda(\rmH^0_\m(M)) \le \ell$, which completes the proof.
\end{proof}

The proof of the implication (2) $\Rightarrow$ (1) of (\ref{5.5}) shows
$$0 \leq \chi_1(\xx;M) \le \displaystyle\sum_{i=0}^{r-1}\binom{r-1}{i} h^i(M)$$ for every amenable partial system $\xx = x_1, \ldots, x_r $ of parameters of $M$ which is a $d$-sequence relative to $M$. Let us add the following.

\begin{Proposition} Suppose that $\rmH^i_\m(M)$ is finitely generated for every  $0 \le  i \le r-1$. 
\begin{itemize}
\item[{\rm (1)}] Suppose that there exists a system of parameters $a_{1}, \ldots, a_{d}$ of $M$ which is a strong $d$-sequence relative to $M$.
Then 
\[ \sup_{n_1, \ldots, n_r >0}\chi_1(a_1^{n_1},\ldots, a_r^{n_r};M) = \displaystyle\sum_{i=0}^{r-1}\binom{r-1}{i} h^i(M).\]  
\item[{\rm (2)}] There always exists a system $x_1, \ldots, x_d$ of parameters of $M$ which is a $d$-sequence relative to $M$ such that $$\chi_1(x_1,\ldots, x_r;M) = \displaystyle\sum_{i=0}^{r-1}\binom{r-1}{i} h^i(M).$$ Therefore  $\max\Gamma (M) = \displaystyle\sum_{i=0}^{r-1}\binom{r-1}{i}h^i(M)$.
\end{itemize}
\end{Proposition}

\begin{proof} (1)
We may assume that $r > 1$ and that the assertion holds true for $r-1$. Let $n >0$ be an integer such that $\m^n\rmH^i_\m(M)= (0)$ for all $0 \le  i \le r-1$. Let $\overline{M} = M/a_{1}^{n}M$. By Lemma \ref{4.4} (1) we then have $h^i(\overline{M}) = h^i(M) + h^{i+1}(M)$ for $0 \le i \le r-2$. Now choose integers $n_2, \ldots, n_r >0$ so that
$$\chi_1(a_2^{n_2}, \ldots, a_r^{n_r};\overline{M}) = \displaystyle\sum_{i=0}^{r-2}\binom{r-2}{i} h^i(\overline{M}).$$ 
Let $\aa= a_1^{n_1},\ldots, a_r^{n_r}$ and $\aa'=a_2^{n_2},\ldots, a_r^{n_r}$.
Then 
\[ \begin{array}{rcl}
{\ds \chi_1(\aa;M) = \chi_1(\aa';\overline{M}) =  \sum_{i=0}^{r-2}\binom{r-2}{i} h^i(\overline{M}) } 
&=& {\ds \sum_{i=0}^{r-2}\binom{r-2}{i} \left[h^i(M)+h^{i+1}(M)\right] } \\
&= &{\ds \sum_{i=0}^{r-1}\binom{r-1}{i} h^i(M).}\\
\end{array}\]

 (2) Let $\fka_i(M) = \operatorname{Ann}_R\rmH_\m^i(M)$ for each $0 \le i \le r-1$ and set $\fka(M) = \displaystyle\prod_{i=0}^{r-1}\fka_i(M)^{\binom{r-1}{i}}.$
Then since the ideal $\fka(M)$ contains a high power of $\m$, by (\ref{ubiquityth}) there is a system $x_1, \ldots, x_d$ of parameters of $M$ which is contained in $\fka(M)$ and is a $d$-sequence relative to $M$. We will show by induction on $r$ that $$\chi_1(x_1, \ldots, x_r;M) =\displaystyle\sum_{i=0}^{r-1}\binom{r-1}{i} h^i(M).$$ 
We may assume that $r > 1$ and that our assertion holds true for $r-1$. Let $\overline{M} = M/x_1M$. We then have, by Lemma \ref{4.4} (1), for each $0 \le i \le r-2$ the  short exact sequence
$$0 \to \rmH_\m^i(M) \to \rmH_\m^{i}(\overline{M}) \to \rmH_\m^{i+1}(M) \to 0,$$
because $x_1\rmH_\m^j(M) = (0)$ for $j = i, i+1$. Hence $h^i(\overline{M}) = h^i(M) + h^{i+1}(M)$ and $\fka_i(M)\fka_{i+1}(M) \subseteq \fka_i(\overline{M})$. Consequently $\fka(M) \subseteq \fka(\overline{M})$. Let $\xx=x_1, \ldots, x_r$ and $\xx'=x_2, \ldots, x_r$.  Then, since $x_2, \ldots, x_d \in \fka(\overline{M})$,  by the hypothesis of induction on $r$, we get 
\[ \begin{array}{rcl}
{\ds \chi_1(\xx;M) = \chi_1(\xx';\overline{M}) 
= \sum_{i=0}^{r-2}\binom{r-2}{i} h^i(\overline{M}) }
&=& {\ds \sum_{i=0}^{r-2}\binom{r-2}{i} \left[h^i(M)+h^{i+1}(M)\right]} \\
&=& {\ds \sum_{i=0}^{r-1}\binom{r-1}{i} h^i(M).}
\end{array}
\]
\end{proof}

\section{The structure of some $\mathbf{j}$-transforms}

The general outline of  $\mathbf{j}$-transforms $\H = \H_{\m}^0(\gr_I(M))$ is still unclear. 
In two cases however,  Buchsbaum 
 and sequentially Cohen-Macaulay modules,  one has a satisfying 
vista.

\subsubsection*{Buchsbaum modules}

Let $(R,\m, k)$ be a Noetherian local ring and $M$  a finite $R$-module of dimension $d \ge 2$. Let $x_1, \ldots, x_{d}$ be a system of parameters of $M$. We fix an integer $0 < r < d$  and set $I =(x_1, \ldots, x_r)$, $\rmG =\mathrm{gr}_I(R)$, $\rmG (M) = \mathrm{gr}_I(M)$, and $$\overline{\rmG} = \rmG/\m \rmG = k[\TT_1, \TT_2, \ldots, \TT_r],$$ where $\TT_i$ denotes the image of $f_i = x_i + I^2$ in $\overline{\rmG}$. 
Let us consider the $\mathbf{j}$-transform $\H = \rmH_\m^0(\rmG (M))$. The goal is  the following.

\begin{Theorem}\label{thm1} Suppose that $M$ is a Buchsbaum $R$-module. Then the following assertions hold true.
\begin{enumerate}
%\item[$(1)$] $\H= (0)$ if and only if $\depth M > r$.
\item[$(1)$] Suppose that $\H \ne (0)$. Then 
\[ \dim \H = \left\{
\begin{array}{ll}
0 & \quad \mbox{if} \quad h^1(M)= h^2(M)= \cdots = h^r(M) = 0, 
\vspace{4mm}\\
r & \quad \mbox{otherwise}.
\end{array}
\right.\]
\item[$(2)$] $\H \cong \bigoplus_{i=0}^r\left[Z_i(i)\right]^{\oplus h^i(M)}$ as a graded $\rmG$-module.
\end{enumerate}
Here $Z_i = \operatorname{Syz}_{\overline{G}}^i(k)$ denotes the $i$-th syzygy module of the residue class field $ k = \overline{\rmG}/[\overline{\rmG}]_+$. The second exact sequence in {\rm Theorem~\ref{jdseqcx}} gives a graded minimal free resolution of the $\overline{\rmG}$-module $\H$, where we identify $\overline{\rmG} = R/\m \otimes_{R} S$ {\rm (cf. Corollary~\ref{vanishingofH})}. 
\end{Theorem}

Let us  begin with the following.

\begin{Lemma}\label{lemma2} Let $\xx=x_{1}, \ldots, x_{r}$ be an amenable partial system of parameters of $M$ that is a $d$-sequence relative to $M$. Let $f_{i}=x_{i}+ I^{2}$ in $\rmG$ for each $i=1, \ldots, r$.
\begin{itemize}
\item[{\rm (1)}] The sequence $f_1, \ldots, f_r$  is a $d$-sequence relative to $\H$, $$\H/(f_1, f_2, \ldots, f_r) \H = [ \H/(f_1, f_2, \ldots, f_r) \H]_0,$$ and $[ \H/(f_1, f_2, \ldots, f_r) \H]_0 \cong \rmH_\m^0(M/IM)$ as an $R$-module.

\item[{\rm (2)}] $(f_1,\ldots, f_i) \rmG(M) \cap \H = (f_1,  \ldots, f_i) \H$ for all $1 \le i \le r$.
\end{itemize}
\end{Lemma}

\begin{proof} (1) Note that $(0):_{\H}f_1 = [(0):_{\H}f_1 ]_0$ and $\rmH (\overline{M}) = \rmH/f_1\rmH$ when $r > 0$.
The assertion follows from Lemma \ref{ind1} and induction on $r$.

 (2) Let $h \in \H$ and assume that $h = f_1 g$ for some $g \in \rmG(M)$. We choose an integer $\ell \gg 0$ so that $\m^\ell \H = (0)$. Then since $\m^\ell h =(0)$, we get  $$\m^\ell g \subseteq (0):_{\rmG(M)}f_1 \subseteq H.$$  Thus $\m^{2\ell} g =(0)$ and hence $g \in \H$. Assume that $i > 1$ and that the assertion holds true for $i -1$. Let $\overline{M} = M/(x_1, \ldots, x_{i -1})M$. Let $h \in (f_1, \ldots, f_i) \rmG(M) \cap \H$.  Then $\overline{h} \in f_i \rmG (\overline{M})\cap \rmH (\overline{M})$, where $\overline{h}$ denotes the image of $h$ in $\rmG (\overline{M})$ under the canonical map $\psi : \rmG (M) \to \rmG (\overline{M})$ and $\rmH (\overline{M})$ denotes  the $\mathbf{j}$-transform of $\overline{M}$ relative to $I$. Because $\rmH (\overline{M}) = \psi (\H)$ by Lemma \ref{ind1} (2), from the case where $i = 1$ it follows that
$$h \in f_i \H + (f_1, \ldots, f_{i -1})\rmG (M).$$
Hence $h - f_i h' \in (f_1, \ldots, f_{i -1})\rmG (M)\cap \H$ for some $h' \in \H$, so that the hypothesis of induction on $i$ implies  $h - f_i h' \in (f_1, \ldots, f_{i -1}) \H$. Thus $h \in (f_1, \ldots, f_{i}) \H$.
\end{proof}

\begin{Proposition}\label{prop3}
$f_1, \ldots, f_r$ is a $\rmG (M)/ \H$-regular sequence.
\end{Proposition}

\begin{proof}
 Let $1 \le i \le r$ be an integer. Let $g \in \rmG (M)$ and assume that $$f_i g \in (f_1, \ldots, f_{i-1})\rmG (M) + H.$$ Then $$f_i g - \sum_{j = 1}^{i-1}f_jg_j \in (f_1, \ldots, f_i)\rmG (M) \cap \H$$ with $g_j \in \rmG (M)$.  Since $ (f_1, \ldots, f_i)\rmG (M) \cap \H = (f_1, \ldots, f_i) \H$ by Lemma \ref{lemma2}, we get
$$f_i g - \sum_{j = 1}^{i-1}f_jg_j  = \sum_{j=1}^{i}f_jh_j$$
with $h_j \in \H$. Hence $f_i (g - h_i) = \sum_{j=1}^{i-1}f_j (h_j + g_j),$
so that $$g - h_i \in (f_1, \ldots, f_{i-1})\rmG (M):_{\rmG (M)}f_i.$$ Let $\overline{M} = M/I_{i-1}M$ and identify $$\rmG (M)/(f_1, \ldots, f_{i-1})\rmG (M) = \rmG (\overline{M}).$$ Then since $(0) :_{\rmG (\overline{M})}f_i \subseteq \rmH (\overline{M})$ and $\rmH (\overline{M})$ coincides with the image of $\H$ in $\rmG(\overline{M})$, we get $$g - h_{i} \in (f_1, \ldots, f_{i-1})\rmG (M) + H.$$ Thus $g \in (f_1, \ldots, f_{i-1})\rmG (M) + \H$, whence 
$f_1,\ldots,f_r$ is a $\rmG (M)/\rmH$-regular sequence. 
\end{proof}

\begin{Corollary}\label{cor4} Suppose that $x_1, \ldots, x_r$ is a  $d^+$-sequence relative to $M$. Then $f_1, \ldots, f_r$ is a $d^+$-sequence relative to $\H$. 
\end{Corollary}

\begin{proof}
 It suffices to check that $f_1, \ldots, f_r$ is a strong $d$-sequence relative to $\H$, which readily follows by definition from the facts that $f_1, \ldots, f_r$ is a strong $d$-sequence on $\rmG(M)$ (\cite[Corollary 5.5]{T}) and that $f_1, \ldots, f_r$ is a $\rmG(M)/\H$-regular sequence (Proposition \ref{prop3}). 
\end{proof}
% \cite[Theorem (2.10)]{GY}

 Let $\rmH_{\rmG_+}^i(*)$ denote the local cohomology functor relative to  the graded ideal $\rmG_+ =(f_1, \ldots, f_r)\rmG$.

\begin{Corollary}\label{cor5} We have the following.
\begin{enumerate}
\item[$(1)$] $\rmH_{\rmG_+}^i( \H) \cong \rmH_{\rmG_+}^i(\rmG (M))$ as a graded $\rmG$-module for all $i < r$.
\item[$(2)$] $[\rmH_{\rmG_+}^r( \H)]_n = (0)$ for all $n > -r$.
\end{enumerate}
\end{Corollary}

\begin{proof}
 Apply $\rmH_{\rmG_+}^i(*)$ to the exact sequence
$$0 \to \H \to \rmG (M) \to \rmG (M)/ \H \to 0$$
of graded $\rmG$-modules and we get the assertion (1), because
$f_1, \ldots, f_r$ is a $\rmG (\H)/ \H$-regular sequence. Assertion (2) follows from the fact that $\mu_\rmG(\rmG_+)\le r$ and $\H = \rmG{\cdot}\rmH_0$.
\end{proof}

When $x_1, \ldots, x_r$ is a $d^+$-sequence relative to $M$, we furthermore have the following.

\begin{Proposition}\label{prop6} Let $I=(\xx) = (x_1, \ldots, x_r )$ be an ideal generated by a $d^+$-sequence relative to $M$. 
\begin{enumerate}
\item[$(1)$] $\rmH_{\rmG_+}^i( \H) = [\rmH_{\rmG_+}^i( \H)]_{-i}$ for all $i < r$.  
\item[$(2)$] Suppose that $\dim \H = r$. Then the $\rmG$-module $\H$ has {\rm FLC},  so that $\H$ is a Cohen-Macaulay $\rmG$-module, if $\rmj_1(\xx;M) = 0$ and $\depth M > 0$.
\end{enumerate}
\end{Proposition}

\begin{proof}
By Corollary \ref{cor5} and \cite[(5.1) Theorem]{Brodmann} we get 
$\rmH_{\rmG_+}^i( \H) \cong \rmH_{\rmG_+}^i(\rmG (M)) = [\rmH_{\rmG_+}^i(\rmG (M))]_{-i},$ whence the assertion follows. The second assertion is  clear, because $\e_1(\mathbf{f};\H) = \rmj_1(\xx;M)$, where $\mathbf{f} = f_1, \ldots, f_r$ (notice that $\H$ is unmixed, if $\depth M > 0$).
\end{proof}

We are in a position to prove Theorem \ref{thm1}

\medskip

 {\bf {\em Proof of Theorem~\ref{thm1}.}} We remember that $x_1, \dots, x_d$ is a standard system of parameters of $M$, since $M$ is a Buchsbaum $R$-module.

 (1) By Theorem \ref{19}
\[ \dim \rmH = r- \min \{0 \le i \le r \mid k_{i}(\xx;M) \ne 0\}\] and by Proposition \ref{4.6FLC} 
\[  k_i(\xx; M) = \sum_{j=1}^{r-i}\binom{r-i -1}{j-1}h^j(M) \quad \mbox{\rm for all}\;\; 0 \leq i \leq r-1. \]
Therefore, if $h^1(M)= h^2(M)= \cdots = h^r(M) = 0$, then 
\[ k_{i}(\xx; M) =0 \quad \mbox{\rm for all}\;\; 0 \leq i \leq r-1, \]
while $k_{r}(\xx; M) = h^{0}(M) \neq 0$ by Theorem \ref{19}, because $\H \neq (0)$.  Hence $\dim \H=0$. If $h^{i}(M) \neq 0$ for some $1 \leq i \leq r$, then $k_{0}(\xx ; M) \neq 0$ by Proposition \ref{4.6FLC}, whence $\dim \H=r$.

 (2)  We may assume that $\H \ne (0)$ and $\dim \H= r$.  Let $S = \overline{\rmG}$ and $\fkM = S_+$. Then $\H$ is a graded $S$-module, because $\fkm \rmH = (0)$ (remember that $\fkm \rmH_\fkm^0(M/IM) = (0)$, since $r < d$). Hence $\H$ is a Buchsbaum $S$-module by Proposition \ref{prop6} and \cite[Proposition 3.1]{G1}. Consequently, by \cite[Theorem (1.1)]{G2} applied to the $S_\fkM$-module $\H_\fkM$, we get an isomorphism
$$(D_0)\ \ \ \H_\fkM \cong \bigoplus_{i=0}^{r-1}(E_i)_\fkM^{\oplus h^i(\H)} \oplus S_\fkM^{\oplus q}$$ of $S_\fkM$-modules, where $q \ge 0$ is an integer.  Taking the graded module $\gr_{\fkn}(*)$  (here $\fkn = \fkM S_\fkM$) of both sides  in $(D_0)$, we get a required direct sum decomposition $$(D)\ \ \ \H \cong \bigoplus_{i=0}^{r-1}E_i^{\oplus h^i(\H)} \oplus S^{\oplus q}$$ of $\rmH$ as a graded $S$-module, because $\gr_{\fkn}(\H_\fkM) = \H$ and $\gr_\fkn((E_i)_\fkM) = E_i$ (remember that $\H = S{\cdot}\H_0$ and $E_i = S{\cdot}[E_i]_0$). On the other hand, since $f_1\rmH_\fkM^i(\rmH) = (0)$ for all $i \ne r$ and $\rmH/f_1\rmH = \rmH(M/x_1M)$ by Lemma \ref{ind1}, we have  the exact sequence
$$ 0 \to \rmH^{i-1}_\fkM(\H) \to \rmH^{i}_\fkM(\rmH(M/x_1M)) \to \rmH^{i+1}_\fkM( \H)(-1) \to 0$$ of graded $S$-modules. The induction on $i$ now shows that $h^i( \H) = h^i(M)$ for all $i < r$,  while we have by decomposition $(D)$ of $\rmH$ that $$\operatorname{rank}_S \H = \sum_{i=1}^{r-1}\binom{r-1}{i-1}h^i(M) + q,$$ since $\operatorname{rank}_SE_i = \binom{r-1}{i-1}$. Let $\Bbb I(\H) = \sum_{i=0}^{r-1}\binom{r-1}{i}h^i(\H)$ be the St{\"u}ckrad-Vogel invariant of $\H$. Therefore,  because $\rmH$ is a Buchsbaum $S$-module, we have
\[ \begin{array}{rcll}
\operatorname{rank}_S \H &=& {\ds \e_0(\TT_1, \ldots, \TT_r; \H) } & \quad \\ &&& \\
&=& {\ds \lambda(\H/(\TT_1, \ldots, \TT_r) \H) - \Bbb I (\H) } & \quad \text{\cite[Proposition 2.6]{SV}}\\ &&& \\
&=& {\ds h^0(M/IM) - \sum_{i=0}^{r-1}\binom{r-1}{i}h^i(M) } & \quad \text{Lemma \ref{lemma2}}\\ &&& \\
&=& {\ds  \sum_{i=0}^{r}\binom{r}{i}h^i(M) -\sum_{i=0}^{r-1}\binom{r-1}{i}h^i(M) } & \quad \text{\cite[Lemma 4.14]{SV}}  \\ &&& \\
&=& {\ds \sum_{i=1}^{r-1}\binom{r-1}{i-1}h^i(M) + h^r(M). } & \\
\end{array}
\]
Thus $q = h^r(M)$, which completes the proof of Theorem \ref{thm1}.
\QED

\begin{Remark} {\rm 
(1) Suppose that $M$ has FLC and our system $x_1, \ldots, x_d$ of parameters of $M$ is standard. If $\H \ne (0)$ and $\dim \H = r$, then $\H$ has FLC and $h^i( \H) = h^i(M)$ for all $i < r$.

(2) It is known that $M$ is a quasi-Buchsbaum $R$-module, that is $\m \rmH_\m^i(M)=(0)$ for all $i \ne d$ if and only if for each $1 \le i \le d$ and for every system $a_1, \ldots, a_d$ of parameters of $M$ such that $a_1,  \ldots, a_d \in \m^2$, one has the equality
$$(a_1,  \ldots, a_{i-1})M:_Ma_i = (a_1, \ldots, a_{i-1})M:_M \m.$$
Hence if $M$ is quasi-Buchsbaum and $x_1, \ldots, x_r \in \m^2$, then $x_1,  \ldots, x_r$ forms a $d^+$-sequence relative to $M$, $\m \rmH_\m^0(M/IM) = (0)$, and $h^0(M/IM) = \sum_{i=0}^r\binom{r}{i}h^i(M).$ Therefore the proof of Theorem \ref{thm1} works also in  this case.

}\end{Remark}

We explore one example.

\begin{Example}\label{ex2}{\rm 
Let $P = k[\![Y_1, Y_2, Y_3, Z_1, Z_2, Z_3]\!]$ be the formal power series ring over a field $k$ and consider the local ring
$$R = P/\left[\fkp \cap \fkq \cap (Y_1^2, Y_2, Y_3, Z_1^2, Z_2,Z_3)\right],$$
where $\fkp = (Y_1, Y_2,Y_3)$ and $\fkq = (Z_1, Z_2, Z_3)$. Then $R$ is a quasi-Buchsbaum but non-Buchsbaum local ring of dimension $3$ such that  $$h^0(R) = h^1(R) = 1\ \ \text{and}\ \  h^i(R) = 0 \ \ \text{if}\ \ i \ne 0,  1, 3$$
(\cite[p. 87, Example]{SV}). For each $1 \le i \le 3$ let $y_i$ and $z_i$ respectively denote the images of $Y_i$ and $Z_i$ in $R$. Let $x_i = y_i + z_i$~$(1 \le i \le 3)$. Then 
$$\lambda(R/(x_1, x_2, x_3)) = 5, \ \ \e_0((x_1, x_2, x_3),R) = 2, \ \ \text{and}\ \  \Bbb I(R) = 3,$$ so that the system $x_1, x_2, x_3$ of parameters of $R$ is standard (\cite[Theorem 2.1]{T}). 

We set  $I = (x_1, x_2)$ and consider the $\mathbf{j}$-transform $\rmH(R)$ of $R$ relative to $I$.  Let $W = \H_{\m}^0(R)$, $\fkM = \m \rmG + \rmG_+$, and $\overline{\rmG} = \rmG/\m \rmG$. We then have the following. Remember that $f_1, f_2$ acts on $\rmH (R)$ as a $d^+$-sequence, since $x_1, x_2$ is a $d^+$-sequence relative to $R$. 

(1) There exists an exact sequence 
$$(\sharp) \ \ 0 \to \overline{\rmG}/\overline{\rmG}_+ \to \rmH(R) \to [\overline{\rmG}_+](1) \to 0$$
of graded $\rmG$-modules. Hence $\rmH(R)$ is a quasi-Buchsbaum $\rmG$-module.

(2) $\m \rmH(R) \ne (0)$ and $\fkM \rmH^0_\fkM(\rmH(R)/f_1\rmH (R)) \ne (0)$. Hence $\rmH(R)$ is not a Buchsbaum $\rmG$-module and exact sequence $(\sharp)$ above cannot be  split.
}\end{Example}

\begin{proof}
(1) Since $R/W = P/\fkp \cap \fkq$, $R/W$ is a Buchsbaum local ring. Hence by Theorem \ref{thm1} $$\rmH(R/W) \cong [\overline{\rmG}_+](1)$$ as a graded $\rmG$-module. The exact sequence now follows from the canonical exact sequence 
$$0 \to W^* \to \rmH(R) \to \rmH(R/W) \to 0$$
of $\mathbf{j}$-transforms, since $W^*\cong \overline{\rmG}/\overline{\rmG}_+$. Therefore, applying the functor $\rmH_\fkM^i(*)$, we get $$\rmH_\fkM^0(\rmH(R)) = W^*\  \text{and}\  \rmH_\fkM^1(\rmH(R)) = \rmH_\fkM^1(\overline{\rmG}_+)(1) \cong \left[\overline{\rmG}/\overline{\rmG}_+\right](1).$$ Thus $\rmH (R)$ is a quasi-Buchsbaum $\rmG$-module.

 (2) We have $\m \rmH_\m^0(R/x_1R) \ne (0)$ and $\m \rmH_\m^0(R/(x_1,x_2)) \ne (0)$, both of which we can check by  a direct computation. Hence exact sequence $(\sharp)$ above  is not split, as $\m \rmH(R) \ne (0)$ (notice that $\m   [\rmH(R)]_0 = \m \rmH_\m^0(R/(x_1, x_2)) \ne (0)$), while $$\fkM  \rmH^0_\fkM (\rmH(R)/f_1\rmH (R)) = \fkM \rmH^0_\fkM (\rmH(R/x_1 R)) \ne (0)$$ because $\m  \rmH_\m^0(R/x_1 R)) \ne (0)$. Thus $\rmH(R)$ is not  a Buchsbaum $\rmG$-module. 
\end{proof}

In Example \ref{ex2},  if we choose a partial system $x_1, x_2$ of parameters of $R$ so that $x_1, x_2 \in \m^2$, we always have $$\rmH (R) \cong \left[\overline{\rmG}/\overline{\rmG}_+\right] \oplus \left[\overline{\rmG}_+\right](1),$$ which shows the structure of $\mathbf{j}$-transforms depends more or less on the choice of partial systems of parameters.

\medskip

\subsubsection*{Sequentially Cohen-Macaulay modules}
Let $(R,\m)$ be a Noetherian local ring and $M$  a finite $R$-module of dimension $d = \dim S M >0$. Let $$\calS = \{\dim N \mid (0) \ne N \subseteq M, \text{an}~{R}\text{-submodule~of}~ M\}.$$ We set $\ell = \sharp \calS$ and write
$\calS = \{d_1 < d_2 < \ldots < d_\ell = d\}.$
Let  $d_0 = 0$. We then have the dimension filtration 
$$\calD_0= (0) \subsetneq \calD_1 \subsetneq D_2 \subsetneq \ldots \subsetneq \calD_\ell = M$$ of $M$, where each $\calD_i~(1 \le i \le \ell)$ is the largest $R$-submodule of $M$ with $\dim \calD_i = d_i$. We put $\calC_i = \calD_i/\calD_{i-1}~(1 \le i \le \ell)$ and assume that $\calC_i$ is a Cohen-Macaulay $R$-module, necessarily of dimension $d_i$, for each $1 \le i \le \ell$. Hence $M$ is a sequentially Cohen-Macaulay $R$-module.

We choose a system $x_1, \ldots, x_d$ of parameters of $M$ so that $$(x_j \mid d_i < j \le d)M \cap \calD_i = (0)$$ for all $1 \le i \le \ell$. Such a system of parameters exists and called a {\it good} system of parameters of $M$. Here we notice that  the condition $(x_j \mid d_i < j \le d)M \cap \calD_i = (0)$ for all $1 \le i \le \ell$ is equivalent to saying that $$(x_j \mid d_i < j \le d)\calD_i = (0)$$ for all $1 \le i \le \ell$, because  $M$ is a sequentially Cohen-Macaulay $R$-module.

Let $0 \le r \le d$ and $I = (x_1, \ldots, x_r)$. We are trying to find what the $\mathbf{j}$-transform $\rmH(M)$ of $M$ relative to $I$ is. The goal is the following. 

\begin{Theorem}\label{thm2}
Let $q = \max \{0 \le i \le \ell \mid d_i \le r\}$. Then $$\rmH(M) \cong \mathrm{gr}_I(\calD_q)$$
as a graded $\rmG$-module, where $\rmG = \mathrm{gr}_I(R)$.  If $q > 0$, that is if $d_1 \le r$, then $\rmH(M) \ne (0)$ and is a sequentially Cohen-Macaulay $\rmG$-module with dimension filtration $\{\mathrm{gr}_I(\calD_i)\}_{0 \le i \le q}$; hence $\dim \rmH(M) = d_q \le r$ and the Hilbert function of $\rmH(M)$ is given by
$$\displaystyle\sum_{k=0}^n\lambda([\rmH(M)]_k) = \displaystyle\sum_{i=1}^q\lambda(\calC_i/\fkq \calC_i)\binom{n+ d_i}{d_i}$$
for all $n \ge 0$, where $\fkq = (x_1,\ldots, x_d)$.
\end{Theorem}

To prove this, we need the following.

\begin{Lemma}\label{lem1} $I^n M \cap \calD_i = I^n\calD_i$ for $0 \le i \le \ell$ and $n \in \Bbb Z$.
\end{Lemma}

\begin{proof}
We may assume that $\ell > 1$ and that the assertion holds true for $\ell -1$. Since $\calD_{\ell - 1}$ is a sequentially Cohen-Macaulay $R$-module with  dimension filtration $\{\calD_i\}_{0 \le i \le \ell -1}$ and the system $x_1, \ldots, x_{d_{\ell-1}}$ is a good system of parameters of $\calD_{\ell-1}$, it suffices to show that 
$$I^nM \cap \calD_{\ell-1} = I^n\calD_{\ell-1}$$
for all $n \ge 0$, which follows from the fact that $\calD_\ell$ is a Cohen-Macaulay $R$-module of dimension $d= d_\ell$.
\end{proof}

Since $I^n\calC_i  =[I^n \calD_i + \calD_{i-1}]/{\calD_{i-1}}$, we have the exact sequence 
$$0 \to I^n\calD_{i-1} \to I^n\calD_i \to I^n \calC_i \to 0$$
of $R$-modules for each $n \in \Bbb Z$, which yields the  exact sequence
$$0 \to \rmG(\calD_{i-1}) \to \rmG(\calD_i) \to \rmG(\calC_i) \to 0$$
of graded $\rmG$-modules, where $\rmG(\calD_i) = \gr_I(\calD_i)$ and $\rmG(\calC_i)=\gr_I(\calC_i)$. Remember that 
$$\rmG(\calD_i) = \mathrm{gr}_J(\calD_i)\ \ \text{and}\ \ \rmG(\calC_i) = \mathrm{gr}_J(\calC_i),$$ where $J = I$ if $r \le d_i$ and $J = (x_1, \ldots, x_{d_i})$ if $r > d_i$. Hence we get the following.

\begin{Lemma}\label{lem3} Let $1 \le i \le \ell$. Then 
\[ \rmH_\m^0(\rmG(\calC_i)) = \left\{
\begin{array}{ll}
\rmG(\calC_i) & \quad \mbox{if} \ \  d_i \le r, \vspace{4mm}\\
(0) & \quad \mbox{if} \ \  r < d_i.
\end{array}
\right.\] 
\end{Lemma}

We are now ready to prove Theorem~\ref{thm2}.

\medskip

{\bf {\em Proof of Theorem~\ref{thm2}.}}
We apply the functor $\rmH^0_\m(*)$ to 
the exact sequences
$$0 \to \rmG(\calD_{i-1}) \to \rmG(\calD_i) \to \rmG(\calC_i) \to 0$$
$(1 \le i \le \ell)$ and get by (\ref{3}) that
$$\rmH(M) \cong \rmG(\calD_q),$$ since $\rmH^0_\m(\rmG(\calC_i)) = (0)$ if $q < i$. Suppose $q > 0$, that is $d_1 \le r$. Then $\{\rmG(\calD_i)\}_{0 \le i \le q}$ gives rise to the dimension filtration of $\rmG(\calD_i)$, since the graded module $\rmG(\calC_i)~(1 \le i \le q)$ is Cohen-Macaulay and $\dim \rmG(\calC_i) = d_i$. Hence $\dim \rmH(M) = d_q \le r$, whose Hilbert function is given by
\begin{eqnarray*}
\displaystyle\sum_{k=0}^n\lambda([\rmH(M)]_k) 
&=& \sum_{k=0}^n\left\{\displaystyle\sum_{i=1}^q\lambda([\rmG(\calC_i)]_k)\right\}\\
&=&\displaystyle\sum_{i=1}^q\lambda(\calC_i/\fkq \calC_i)\binom{n+ d_i}{d_i}
\end{eqnarray*}
for all $n \ge 0$.
\QED

We close this paper with the following.

\medskip

\begin{Example}\label{6.12}{\rm Let $(R,\fkn)$ be a regular local ring of dimension $d+1~(d \ge 2)$ and let $X, Y_1, Y_2, \ldots, Y_d$ be a regular system of parameters of $R$. We set $\fkp = (Y_1, Y_2, \ldots, Y_d)$, $M = R/[(X) \cap \fkp]$, and $V = R/\fkp$. Then $\dim M = d$ and the dimension filtration of $M$ is given by
$$\calD_0 = (0) \subsetneq \calD_1 = Rx \subsetneq \calD_2 = M,$$
where $x$ is the image of $X$ in $M$ and $V \cong Rx$ as an $R$-module, so that $M$ is a sequentially Cohen-Macaulay $R$-module. Let $\xx = x_1, \ldots, x_r~(0 \le r \le d)$ be an arbitrary partial system of parameters of $M$ and set $I = (\xx)$. We then have
\[ \rmH(M) \cong \left\{
\begin{array}{ll}
(0) & \quad \mathrm{if} \quad I \subseteq \fkp,
\vspace{4mm}\\
\gr_I(V)& \quad \mathrm{if}  \quad I \not\subseteq \fkp \ \ \mathrm{and}  \ \  r < d,
\vspace{4mm}\\
\gr_I(M) & \quad \mathrm{if} \quad r = d.
\end{array}
\right.\]
Notice that  when $0 < r$ and $I \subseteq \fkp$, $\xx$ forms a $d$-sequence relative to $M$, but $\xx$ is not amenable for $M$, because $(0):_MI = Rx$ and $\dim Rx = 1$.
}\end{Example}

\begin{proof} Let $N = Rx$, $C = R/(X)$, and we consider the exact sequence
$$0 \to N \to M \to C \to 0$$
of $R$-modules. Suppose  $I \subseteq \fkp$. Then  $r < d$ and $\xx$ is a $C$-regular sequence. Therefore since $I^{n+1}M \cap N = I^{n+1}N = (0)$ for all $n \ge 0$, we get
the exact sequence
$$0 \to N \to M/I^{n+1}M \to C/I^{n+1}C \to 0$$
of $R$-modules, which shows $\rmH_\fkm^0(M/I^{n+1}M) = (0)$, because $C/I^{n+1}C$ is a Cohen-Macaulay $R$-module of dimension $d - r > 0$. Thus $\rmH(M) = (0)$.

Suppose $I \not\subseteq \fkp$. Then $IV \ne (0)$ and $0 < r$. We choose $1 \le i \le r$ so that $IV = x_iV$ (this choice is possible, because $V$ is a discrete valuation ring). Let $1 \le j \le r$.  Then $x_j \in x_iR + \fkp$. We write $x_j = x_ic_j + y_j$ with $c_j \in R$ and $y_j \in \fkp$. Hence $I = (x_i) + (y_j \mid 1 \le j \le r, \ j \ne i)$. Because $y_j \in \fkp$ for all $j$ and $x_iN \ne (0)$, the sequence $x_i, \{y_j\}_{1 \le j \le r, j \ne i}$ is extended to a good system of parameters of $M$. Therefore, if $r < d$, $\rmH(M) \cong \gr_I(V)$ by Theorem \ref{thm2}, while $\rmH(M) = \gr_I(M)$, if $r = d$.
\end{proof}

\end{document}